\documentclass[a4,11pt]{amsart}

\usepackage[usenames]{color}
\usepackage{amsfonts}
\usepackage{amssymb}
\usepackage[normalem]{ulem}

\usepackage{soul}
\usepackage{amsmath,amssymb,amsthm,amsfonts,enumerate,url,mathrsfs,tikz,graphicx}
\usepackage{hyperref}

\newcommand\cev[1]{\overleftarrow{#1}}

\usetikzlibrary{automata,positioning}
	
\usepackage[utf8]{inputenc}

\usepackage{verbatim}
\usepackage{array}

\usepackage{bbm,dsfont}
\usepackage[all]{xy}

\newtheorem{lemma}{Lemma}[section]
\newtheorem{corollary}[lemma]{Corollary}
\newtheorem{definition}[lemma]{Definition}
\newtheorem{theorem}[lemma]{Theorem}
\newtheorem{remark}[lemma]{Remark}
\newtheorem{proposition}[lemma]{Proposition}
\newtheorem{ex}[lemma]{Example}

\numberwithin{equation}{section}
\numberwithin{lemma}{section}

 \def\mG{\mathsf{G}}

 \def\mV{\mathsf{V}}
 \def\mE{\mathsf{E}}

 \def\mv{\mathsf{v}}
 \def\me{\mathsf{e}}
 \def\mw{\mathsf{w}}

 \def\mf{\mathsf{f}}

\newcommand{\Z}{\mathbb{Z}}

\newcommand{\R}{\mathbb{R}}
\newcommand{\C}{\mathbb{C}}
\newcommand{\N}{\mathbb{N}}

\newcommand{\A}{{A}}
\newcommand{\U}{\ \mathcal{U}}

\newcommand{\V}{\mathcal{V}}

\renewcommand{\L}{\mathcal{L}}
\newcommand{\T}{\mathcal{T}}

\renewcommand{\d}{\mathrm{d}}
\renewcommand{\Re}{\operatorname{Re}}
\renewcommand{\Im}{\operatorname{Im}}
\newcommand{\fra}{\mathfrak{a}}
\newcommand{\frb}{\mathfrak{b}}

\renewcommand{\mid}{\, \vert \,}
\newcommand{\MaxR}{\textit{MR}\,}

\DeclareMathOperator{\essinf}{ess\;inf}
\DeclareMathOperator{\id}{id}
\DeclareMathOperator{\sgn}{sgn}

\DeclareMathOperator{\diag}{diag}
\DeclareMathOperator{\Id}{Id}
\DeclareMathOperator{\Ker}{Ker}

\DeclareMathOperator{\Formm}{\textsc{Form}}

\newcommand\restr[2]{{#1}{_{|#2}}}
\newcommand\HL[1]{{\color{blue} HL: #1}}

	\title[Ultracontractivity and Gaussian bounds for evolution families]{Ultracontractivity and Gaussian bounds for evolution families associated with non-autonomous forms}



\thanks{The first author was supported by the Deutsche Forschungsgemeinschaft (Grant LA 4197/1-1). The second author was partially supported by the Deutsche Forschungsgemeinschaft (Grant 397230547).}

\author{Hafida Laasri}
\address{Hafida Laasri, Arbeitsgruppe Funktionalanalysis, Fakult\"at Mathematik und Informatik, Universit\"at Wuppertal, 42119 Wuppertal, Germany}
\email{laasri@uni-wuppertal.de}

\author{Delio Mugnolo}
\address{Delio Mugnolo, Lehrgebiet Analysis, Fakult\"at Mathematik und Informatik, Fern\-Universit\"at in Hagen, 58084 Hagen, Germany}
\email{delio.mugnolo@fernuni-hagen.de}
	
\begin{document}
	
\begin{abstract}

We develop a variational approach in order to study qualitative properties of non-autonomous parabolic equations. Based on the method of product integrals, we discuss long-time behavior, invariance properties, and ultracontractivity of evolution families in Hilbert space. Our main results give sufficient conditions for the heat kernel of the evolution family to satisfy Gaussian-type bounds. Along the way, we study  examples of non-autonomous equations on graphs, metric graphs, and domains.

\end{abstract}

\maketitle

\section{Introduction}\label{s1}

Non-autonomous evolution equations are partial differential equations in which relevant coefficients of the differential operator and/or in the boundary conditions are time-dependent, thus allowing for underlying models that are variable over time.

In the autonomous case (i.e., evolution equations with time-independent coefficients), well-posed\-ness is equivalent to generation of a semigroup in a suitable Banach space; in comparison, the theory that describes well-posedness of non-autonomous problems on general Banach spaces is more rudimentary. If the coefficients of a non-autonomous equation are piecewise constant, then one may find a solution by following the orbit of the semigroup governing a given problem as long as the coefficients stay constant; then ``freeze'' the system; use the final state as an initial condition for a new evolution equation with new (constant) coefficients, and so on: this boils down to consider the composition of a finite numbers of semigroups.

A theory originally developed by J.-L.\ Lions shows that well-posedness in Hilbert space can be proved under much weaker assumptions, most notably mere measurability of the time dependence, provided the problem has a nice variational structure: this is typically the case if the differential equation is parabolic at any given time. By adapting the setting of (time-independent) bounded elliptic forms it is thus possible to show that the equation has a solution that is, in particular, continuous in time. 
This motivates the study of \textit{non-autonomous forms}, a topic which has received much attention in the last decade: we mention among others~\cite{AreDieLaa14,AreDie18,AreDieFac17,Ouh15,Fac17}. All these articles are chiefly devoted to study properties of solutions of partial differential equations, with a focus on maximal regularity issues and hence allowing for inhomogeneous terms.

Our main aim in this paper is to develop an abstract theory with a more operator-theoretical flavor. Indeed, Lions' result paves the way to the possibility of defining an \textit{evolution family} (or \textit{evolution system}, or \textit{propagator}),
 i.e., a family of operators $U(\cdot,s)$ mapping each initial data
\[
u(s)=x\in H
\]
to the orbit of the solution 
\[
\dot{u}(t)+\A(t)u(t)=0 \quad \hbox{a.e. on}\ [s,T]\ .
\]
Because the initial condition may well be imposed at instants $s\ne 0$, this actually define a two-parameter family 
\[
\mathcal U:=(U(t,s))_{(t,s)\in \overline{\Delta}}
\]
 of bounded linear operators on $H$ by $U(t,s)x:=u(t)$, where  $
  \Delta:=\{(t,s)\in (0,T)^2:s< t\}.$
Some good compendia on such evolution families are~\cite[Chapt.~7]{Tan79}, \cite[Chapt.~5]{Paz83},~\cite[Chapt.~7]{Fat83},~\cite[Section~VI.9]{EngNag00}, or the monograph \cite{ChicLat99}.

 The tumultuous development of Hilbert space methods, and especially the theory of Dirichlet forms, have been fruitful also in the non-autonomous environment: a theory of non-autonomous Dirichlet forms has been recently introduced in~\cite{AreDieOuh14}. If $A(t)\equiv A$, the above abstract Cauchy problem is autonomous and its solution is simply given by
\[
u(t)=U(t,s)x:=e^{-(t-s)A}x\ ;
\]
hence the findings in~\cite{AreDieOuh14} can be regarded as a strict generalization of the classical theory of Markovian operators and Dirichlet forms represented e.g.\ in~\cite{FukOshTak10}.
Our goal is to complement these results, thus setting up a non-autonomous variational program analogous to the autonomous one outlined in classical monographs like~\cite{Ouh05}: among other things we study  extrapolation to $L^p$-spaces, ultracontractivity, or Gaussian-type bounds on integral kernels of evolution families. 

It should be mentioned that ultracontractivity and kernel estimates have been observed already in~\cite{Aro68,Dan00} for specific instances of parabolic non-autonomous equations; in particular, Aronson observed in~\cite{Aro67} that the fundamental solution $(t,s;x,y)\mapsto \Gamma(t,s;x,y)$ of a certain class of non-autonomous diffusion equations in (domains of) $\R^d$ satisfies
\begin{equation}\label{eq:nonaut-gauss}
\Gamma(t,s;x,y)\le K\ G(t-s;x-y)
\end{equation}
where $(t,x)\mapsto G(t,x)$ is the Gaussian kernel that yields the fundamental solution of the (autonomous) heat equation on $\R^d$. Analogous Gaussian bounds have ever since been proved for integral kernels of semigroups generated by large classes of second-order elliptic operators, possibly with complex coefficients~\cite{Ouh04}; 
in the non-autonomous case, Aronson's original findings have been extended to operators on domains 
in~\cite{Aro68,Dan00}.

In this paper we are going to introduce a general approach, based on the so-called Davies' Trick, to prove Gaussian bounds for heat kernels of evolution families that govern non-autonomous parabolic equations. Inspired by some techniques introduced in~\cite{Dan00,Ouh04}, we show the applicability of our methods by showing that a large class of elliptic operators with complex-valued, bounded measurable coefficients are associated with evolution families that satisfy Gaussian bounds, thus extending the main results in~\cite{Ouh04} to the non-autonomous setting.

Our approach will heavily rely upon the method of \textit{product integrals}, whose historical evolution is thoroughly discussed in \cite{Sla07},~ \cite[\S~7.10]{Fat83} and whose scope has been extended to non-autonomous forms with measurable dependence on time in~\cite{ElmLaa16, SanLaa15}. We adapt it to our present setting, thus deriving in Theorem~\ref{convergence forte de solution approchee} a version that we will use over and over again in different contexts throughout this paper. The method of product integrals proves especially efficient when it comes to discuss long-time behavior of evolution families. While there is already a wide literature devoted to this topic, see e.g. the survey in~\cite{Sch04}, our setting allows us to provide conditions that are very easy to check in many concrete cases.   While all strongly continuous semigroups are exponentially bounded, this is not the case for general evolution families, cf.~\cite[\S~VI.9]{EngNag00}: unlike in the general case, though, evolution families associated with non-autonomous forms are always exponentially bounded and we give sufficient criteria for exponential stability and for uniform convergence towards equilibrium; in~\cite[\S~5]{AreDieKra14}, only strong convergence was studied.


The present paper is organized as follows. After describing our mathematical framework in Section~\ref{s2}, in Section~\ref{sec:lattice} we present sufficient conditions that enforce qualitative properties based on the lattice structure of $L^2$-spaces and, using Perron--Frobenius-type arguments, we discuss long-time behavior of evolution families.




Gaussian-type bounds are shown to depend on ultracontractivity properties of certain operator families related to $\mathcal U$. This approach requires, in turn, suitable common bounds in $L^p$-norm, uniformly on all compact subsets of $\Delta$. Inspired by similar criteria in the autonomous setting we show that efficient conditions based on Sobolev-type inequalities can enforce such bounds.
In Section~\ref{sec:ultra} we develop a theory of ultracontractive evolution families: a technical difficulty we face is related to the failure of self-adjointness of evolution families, a phenomenon that typically occurs even when all operators $A(t)$ are self-adjoint.
We take over an idea from~\cite{Dan00} and circumvent this problem by studying some non-autonomous form associated with a tightly related backward evolution equation.

It has been known since~\cite{Dav87} that ultracontractivity is an important, but not yet sufficient ingredient to prove Gaussian bounds for semigroups. In Section~\ref{sec:gaussian} we present different conditions that imply  Gaussian bounds for evolution families.
In particular, our approach allows us to show Gaussian bounds for the evolution family associated with a large class of elliptic operators, thus generalizing the pioneering results in~\cite{Aro68,Dan00}.

Several applications are reviewed in Section~\ref{sec:appl}: we discuss well-posedness and qualitative properties of dynamical systems on undirected graphs tightly related to the theory of \emph{dynamic (positive) graphs} discussed in~\cite{Sil08} as well as models of Black--Scholes-types equations with time-dependent volatility~\cite{Hes93}; we extend the kernel estimates in~\cite{Mug07} to more general non-autonomous diffusion equations on possibly infinite networks; and finally, we prove Gaussian bounds for the heat kernel for a large class of elliptic operators with time-dependent, possibly complex coefficients, thus deducing the main results in~\cite{Dan00,Ouh04} as special cases.

\section{Evolution families: Notations and preliminary results\label{s2}}
Throughout this paper $H$ is a separable, complex Hilbert space and $V$ is a further complex Hilbert space that is densely and continuously embedded into $H$. Let $V'$ denote the antidual of $V$ with respect to the pivot space $H$; the duality between $V'$ and $V$ is denoted by $\langle ., . \rangle$. 
We also denote by $(\cdot \mid \cdot)_V$ and $\|\cdot\|_V$ the scalar product and the norm
on $V$, respectively; and by $(\cdot \mid \cdot)$ and $\|\cdot\|$ the corresponding quantities in $H$.

We fix $T\in ]0,\infty[$ and consider a time-dependent family $(a(t))_{t\in [0,T]}$ of mappings such that $a(t;\cdot,\cdot): V\times V \to \C$ is for all $t\in [0,T]$ a sesquilinear form and 
\begin{align}
&\label{measurability} [0,T]\ni t\mapsto a(t;u,v)\in \mathbb C \hbox{ is measurable}\qquad \hbox{for all }u,v\in V;
\end{align}
and such that furthermore there exist constants $M, \alpha> 0$ and $\omega\geq 0$ such that the boundedness and $H$-ellipticity estimates
\begin{align}
\label{boundedness}|a(t;u,v)|\leq M\|u\|_V\|v\|_V\quad &\hbox{for a.e }t\in[0,T]\hbox{ and }u,v\in V,\\ 
\label{ellipticity}\Re a(t;u,u)+\omega\|u\|^2_{H}\geq\alpha\|u\|_V^2\quad &\hbox{for a.e }t\in[0,T]\hbox{ and }u\in V,
\end{align}
hold.
In what follows we call such a family $\fra:=(a(t))_{t\in [0,T]}$ \textit{ bounded $H$-elliptic non-autonomous form}: 
following~\cite{AreDie18} we denote by $\Formm([0,T];V,H)$ the class of all such forms.

By the Lax--Milgram theorem, for each $t\in[0,T]$ there exists an \textit{operator associated} with $a(t,\cdot,\cdot)$, i.e., an isomorphism $\A(t):V\to V^\prime$ such that
\[
\langle \A(t) u, v \rangle = a(t,u,v)\qquad \hbox{ for all }u,v \in V:
\]
 accordingly we refer to the family $(A(t))_{t\in [0,T]}$ as the \textit{operator family} associated with $\fra:=(a(t))_{t\in [0,T]}$.
 
Regarded as an unbounded operator with domain $V$, $-\A(t)$ generates a holomorphic semigroup on $V'$, and in fact by~\cite[Thm.~7.1.5]{Are06} on $H$ too, since $a(t)$ is for all $t$ a bounded, $H$-elliptic sesquilinear form: with an abuse of notation we denote its generator -- the part of $-\A(t)$ in $H$ -- again by $-\A(t)$, and the semigroup by 
\[
\mathcal T_t:=\{e^{-r\A(t)}\mid r\ge 0\} .
\]
Hence, for each fixed $t,s\in [0,T]$ the Cauchy problem
\[
\begin{split}
\dot{u}(r)+\A(t)u(r)&=0, \qquad r\in [s,T],\\
u(s)&=x\in H,
\end{split}
\]
is well-posed, its solution being given by $u(r,x):=e^{-(r-s)\A(t)}x$. However, we are rather going to focus on the non-autonomous Cauchy problem
\begin{equation}\label{evolution equation u(s)=x}
\begin{split}
\dot{u}(t)+\A(t)u(t)&=0, \qquad t\in [s,T],\\
u(s)&=x\in H.
\end{split}
\end{equation}

In order to introduce the main objects of our investigations, let $\fra\in \Formm([0,T];V,H)$. A classical well-posedness theorem by J.-L.\ Lions~\cite[§§~XVIII.3.2--3]{DauLio92} states that for each $s\in[0,T[$ and each $x\in H$
\eqref{evolution equation u(s)=x}
admits a unique solution $u$ in the \textit{maximal regularity space}
\[
MR(V,V'):=MR(s,T;V,V'):=L^2(s,T;V)\cap H^1(s,T;V').
\] 
It is well-known that 
$\MaxR (V,V')$ is continuously embedded into $C([s,T];H)$, see e.g.~\cite[Prop.~III.1.2]{Sho97}: this allows us to introduce a family of linear operators by
 \begin{equation}\label{evolution family}
U(t,s):H\ni x\mapsto U(t,s)x:=u(t)\in H,\qquad (t,s)\in\overline{\Delta},
\end{equation}
where here and in the following we adopt the notation
\[
 \Delta:=\{(t,s)\in (0,T)^2:s< t\}.
\]
and $u$ is the unique solution of the Cauchy problem~\eqref{evolution equation u(s)=x} in $\MR(V,V')$. Letting $X:=V'\hbox{ and } D:=V$ (whence in particular $Tr=H$) in~\cite[Prop.~2.3 and Prop.~2.4]{AreChiFor07}  it now follows that $\U:=(U(t,s))_{(t,s)\in \overline{\Delta}}$ is a \emph{strongly continuous evolution family} on $H$, i.e., the following properties hold:
 \begin{enumerate}[(i)]
 	\item $U(s,s)=\Id_H
 	$ for all $s\in [0,T]$,
 	\item $U(t,s)= U(t,r)U(r,s)$ for all $0\leq s\leq r\leq t\leq T$,
 	\item $(t,s)\mapsto U(t,s)x$ is for all $x\in H$ continuous from $\overline{\Delta}$ into $H$.
 	\end{enumerate}
If in fact $\fra\in\Formm([0,\infty[;V,H)$, then arguing as above we deduce that
\begin{equation*}
\begin{split}
\dot{u}(t)+\A(t)u(t)&=0, \qquad t\in [s,\infty[,\\
u(s)&=x\in H,
\end{split}
\end{equation*}
has for all $s>0$ and all $x\in H$ a unique solution $u\in L^2_{Loc}(s,\infty;V)\cap H^1_{Loc}(s,\infty;V')$, hence $u\in C([0,\infty[;H)$: this defines a strongly continuous evolution family $(U(t,s))_{0\le s\le t<\infty}$.

In the following we will refer to $\U:=(U(t,s))_{(t,s)\in \overline{\Delta}}$ as
the \textit{evolution family associated with the non-autonomous form $\fra$} or \textit{with the operator family $(\A(t))_{t\in [0,T]}$}.

 

\begin{remark}\label{remark-rescaling} 
A non-autonomous form is called \emph{coercive} if~\eqref{ellipticity} is satisfied with $\omega=0$. Now, $\fra$ satisfies~\eqref{ellipticity} if and only if the form $\fra_{\omega}$ given by
\begin{equation*}\label{eq:frao}
	\fra_{\omega}(t;u,v):=a(t;u,v)+\omega (u\mid v)
\end{equation*}
is coercive:
because $\fra_{\omega}\in \Formm([0,T];V,H)$ in its own right, it is associated with an evolution family.
Moreover, $u$ is a solution of class $MR(V,V')$ of~\eqref{evolution equation u(s)=x} if and only if $v:=e^{-\omega(.-s)}u$ is a solution of class $MR(V,V')$ of
	\begin{equation*}
	\begin{split}
	\dot{v}(t)+(\omega+\A(t))v(t)&=0, \qquad t\in [s,T],\\
 v(s)&=x .
 \end{split}
	\end{equation*}
	Thus, the evolution family associated with $\fra_{\omega}$ is simply obtained by rescaling, i.e.,
	\begin{equation}\label{eq:rescaled}
	U_{\omega}(t,s):=e^{-\omega(t-s)}U(t,s),\qquad (t,s)\in \overline{\Delta}.
	\end{equation}
\end{remark}

The earliest well-posedness results for~\eqref{evolution equation u(s)=x} were obtained by Kato based on an approximation method based on the theory of product integrals under strong regularity assumptions on the dependence $t\mapsto a(t)$. Kato's approach has been extended to non-autonomous form of class $\Formm([0,T];V,H)$ in \cite{LaaElm13, ElmLaa16}. We sketch the construction of evolution families proposed in \cite{SanLaa15} for the sake of self-containedness, since we are going to use it repeatedly in the next sections.

Let $\Lambda=(\lambda_0,\ldots,\lambda_{n+1})$ be
a partition of $[0,T]$, i.e., $0=\lambda_0<\lambda_1<\ldots<\lambda_{n+1}=T$. Let $(a_k)_{k\in \N}$ be a family of sesquilinear forms defined by
\begin{equation}\label{eq:form-moyen integrale}
a_k:V\times V\ni (u,v)\mapsto \frac{1}{\lambda_{k+1}-\lambda_k}
\int_{\lambda_k}^{\lambda_{k+1}}a(r;u,v) \d r\in \C,\qquad k=0,1,\ldots,n.
\end{equation}
 All these forms lie in $\Formm([0,T];V,H)$ with constants $M$, $\alpha$, and $\omega$. The associated operators $\A_k\in \L(V,V')$ are given by
\begin{equation}\label{eq:op-moyen integrale}
\A_k:V\ni u\mapsto \frac{1}{\lambda_{k+1}-\lambda_k}
\int_{\lambda_k}^{\lambda_{k+1}}\A(r)u \d r\in V',\qquad k=0,1,\ldots,n.
\end{equation}

The mapping $\A(\cdot):[0,T]\to \L(V,V')$ is strongly measurable by Pettis' Theorem~\cite[Thm.~1.1.1]{AreBatHie01}
 since $t\mapsto \A(t)u$ is weakly measurable
 and $V'$ is assumed to be separable. On the other hand, $\|\A(t)u\|_{V'}\leqslant M \|u\|_V$ for 
all $u\in V$ and a.e. $t\in [0,T].$ Thus $[0,T]\ni t\mapsto \A(t)u\in V'$ is Bochner integrable for all $u\in V.$ Hence the integrals in~\eqref{eq:form-moyen integrale} 
and~\eqref{eq:op-moyen integrale} are well defined.

Next, consider the bounded $H$-elliptic non-autonomous form $\fra_\Lambda:=(a_\Lambda(t))_{t\in[0,T]}$ defined by
 \begin{equation}\label{form: approximation formula1}
 a_\Lambda(t;\cdot,\cdot):V\times V\ni (u,v)\mapsto \begin{cases}
 a_k(u,v)&\hbox{if }t\in [\lambda_k,\lambda_{k+1}[\\
 a_n(u,v)&\hbox{if }t=T\ .
 \end{cases}
\end{equation}
Its associated time-dependent operator family ${\mathcal A}_\Lambda:=(\A_\Lambda(t))_{t\in [0,T]}\subset \L(V,V')$ is given by 
 \begin{equation}\label{form: approximation formula2}
 \A_{\Lambda}(t):=\begin{cases}
 \A_k&\hbox{if }t\in [\lambda_k,\lambda_{k+1}[\\
 \A_n &\hbox{if }t=T\ .
 \end{cases}
\end{equation}
 For each $k=0,1,\ldots,n$ we denote by $\mathcal T_k:=\{e^{-rA_k}\mid r\ge 0\}$ the $C_0$-semigroup generated by $-\A_k.$ For each  $a,b\in [0,T]$ such that
\begin{equation}\label{eq:partition}
\lambda_{m-1}\leq a<\lambda_m<\ldots <\lambda_{l-1}\leq b<\lambda_{l}
\end{equation}
 we define the operator families $\mathcal U_\Lambda:=(U_\Lambda(t,s))_{(t,s)\in \Delta}\subset \mathcal{L}(V')$ by
\begin{equation}\label{promenade1}
 U_\Lambda (b,a)= e^{-(b-\lambda_{l-1})A_{l-1}}e^{-(\lambda_{l-1}-\lambda_{l-2})A_{l-2}}\ldots e^{-(\lambda_{m-1}-\lambda_{m-2})A_{m-2}} e^{-(\lambda_{m}-a)A_{m-1}},
 \end{equation}
 and for $\lambda_{l-1}\leq a\leq b<\lambda_{l}$ by 
 \begin{equation}\label{promenade2}
 U_\Lambda (b,a)= e^{-(b-a)A_l}.
 \end{equation}
 Remark that $\mathcal U_\Lambda$ defines an evolution family on $H$ (as well as on $V'$ and $V$), since all semigroups $\mathcal T_k$ consist of bounded linear operators on $H$. Additionally, one sees that the conditions \eqref{measurability}--\eqref{ellipticity} are satisfied by the forms $\fra_\Lambda$, too. Moreover, for all $x\in H$ the function $u_\Lambda(\cdot):= U_\Lambda(\cdot,s)x$ is the unique solution of class $MR(s,T;V,V')$ of the problem
 \begin{equation}\label{Cauchy discretisation}
 \begin{split}
\dot{u}_\Lambda(t)+\A_{\Lambda}(t)u_\Lambda(t)&=0, \qquad t\in [s,T],\\
u_\Lambda(s)&=x\ .
\end{split}
 \end{equation}


A similar approximation scheme was introduced in~\cite{ElmLaa16} in the more general context of inhomogeneous non-autonomous problems; several convergence results could be deduced there, depending on conditions satisfied by the non-autonomous form. In the language of evolution families, we can paraphrase Proposition 3.1 in~\cite{ElmLaa16} and state the following.

\begin{theorem}\label{convergence forte de solution approchee} 
Let $\fra\in \Formm([0,T];V,H)$ and let $\U$ and $\U_\Lambda$ be the evolution families associated with $\fra$ and $\fra_\Lambda$, respectively. Then 
\[
\lim\limits_{|\Lambda|\to 0}U_\Lambda(t,s)=U(t,s)\qquad\hbox{ for all }(t,s)\in\Delta
\]
in the strong operator topology of $\mathcal L(H)$.
\end{theorem}
 The proof of Theorem  \ref{convergence forte de solution approchee}  is very similar to that of  \cite[Prop.~3.1]{ElmLaa16} and will be omitted. 

The product integral method can be applied to deduce two results about the long-time behavior of evolution families.
The following assertion about quasi-contractivity is similar to~\cite[Prop.~2.1]{Laa18}, whereas \textit{strong} stability was proved in a special case in~\cite[Thm.~5.4]{AreDieKra14}. 

\begin{proposition}\label{prop:ues}
Let $\fra\in \Formm([0,T];V,H)$. Then the associated evolution family $\mathcal U$ is \emph{quasi-contractive}, i.e., 

\[
\|U(t,s)\|_{\mathcal L(H)}\le e^{\int_s^t \omega(r)\ dr}\qquad \hbox{ for all }(t,s)\in \overline{\Delta},
\] 
for some $\omega\in L^1(0,T)$ such that
\begin{equation}\label{eq:accretivity}
	\Re a(t;u,u)+\omega(t)\|u\|^2_{H}\geq 0\qquad \hbox{for a.e. }t\in[0,T]\hbox{ and }u\in V.
	\end{equation}
 If in particular $\fra\in\Formm([0,\infty[;V,H)$, $\omega\in L^1_{loc}([0,\infty[)$, and $\limsup\limits_{t\to\infty}\frac{1}{t-t_0}\int_{t_0}^{t} \omega(r)\ dr=\Omega<0$ for some $t_0$, then $\U$ is \emph{uniformly exponentially stable}, i.e., 
\[
\|U(t,s)\|_{\mathcal L(H)}\le M_{t_0} e^{(t-t_0)\Omega}\qquad \hbox{for some }M_{t_0}\hbox{ and all }t>t_0.
\]
\end{proposition}

In view of Proposition~\ref{prop:ues} and Remark~\ref{remark-rescaling} and up to scalar perturbations of $(A(t))_{t\in [0,T]}$, we will thus often assume without loss of generality that the associated evolution family $\mathcal U$ is contractive.

\begin{proof}
Let $[a,b]\subset [0,T]$, $\Lambda$ be a partition of $[a,b]$ as in~\eqref{eq:partition} and consider the discretized evolution family $\U_\Lambda$. Let $\omega_k\in \R$ be defined by 
\begin{equation}\label{eq:omega-moyen integrale}
\omega_k:=\frac{1}{\lambda_{k+1}-\lambda_k}
\int_{\lambda_k}^{\lambda_{k+1}}\omega(r) \d r, \qquad k=0,1,\ldots,n.
\end{equation}
Then by definition of $a_k$ in~\eqref{eq:form-moyen integrale}, \eqref{eq:accretivity} implies 
\[
\Re a_k(u,u)+\omega_k\|u\|^2_{H}\geq 0\qquad\hbox{for all }u\in V\hbox{ and }k=0,1,\cdots,n,
\]
 hence the associated semigroup satisfies
\begin{equation}\label{eq2:expStability} 
\|e^{-rA_k }\|_{\L(H)}\leq e^{r\omega_k} \qquad \hbox{for all }k=0,1,\ldots,n,\ r\ge 0.
\end{equation}
Now we obtain from \eqref{promenade1}
\begin{equation}\label{eq3:expStability} 
\|U_\Lambda(a,b)\|_{\L(H)} \leq \prod_{k=0}^n e^{\int_{\lambda_{k-1}}^{\lambda_{k}}\omega(r)\d r}= e^{\int_{a}^{b}\omega(r)\d r}.
\end{equation}
The claim now follows from Theorem \ref{convergence forte de solution approchee} and Fatou's Lemma.

The second assertion follows by observing that
\[
\|U(t,s)x\|\le \|U(t,t_0)\| \|U(t_0,s)x\|\le e^{(t-t_0)\Omega}\|U(t_0,s)\|\|x\|=:M_{t_0}e^{(t-t_0)\Omega}\|x\|
\]
for all $0\le s\le t<\infty$ and $x\in H$.
\end{proof}


\section{Invariance Properties}\label{sec:lattice}

Let us discuss invariance of a given subset $C$ of $H$ under $\U$, i.e., whether $u(s)\in C$ implies that the solution $u(t)$ of~\eqref{evolution equation u(s)=x} lies in $C$ for any $(t,s)\in \Delta$. The following criterion is known: it combines~\cite[Thm.~4.1]{SanLaa15} with an extension to non-accretive forms~\cite[Thm.~2.1]{ManVogVoi05} of Ouhabaz' classical invariance criterion~\cite[Thm.~2.1]{Ouh96}.

\begin{proposition}\label{cor:easytho}
Let $\fra\in \Formm([0,T];V,H)$. Let $C$ be a closed convex subset of $H$ and denote by $P$ the projector of $H$ onto $C$. Consider the following assertions:
\begin{enumerate}[(i)]
\item $C$ is invariant under the semigroup $\mathcal T_t$ associated with $a(t)$ for a.e.\ $t\in [0,T]$;
\item $C$ is invariant under the evolution family $\U$.
\item $Pu\in V$ and $\Re a(t;Pu,u-Pu)\ge 0$ for all $u\in V$ and a.e.\ $t\ge 0$;
\end{enumerate}
Then $(i)$ is equivalent to $(iii)$ and both imply $(ii)$.
\end{proposition}

 The implication $(iii)\Rightarrow(ii)$ has been proved in~\cite[Thm.~2.2]{AreDieOuh14} in the more general case of inhomogeneous equations. Special instances of the same assertion have been obtained in~\cite[\S~3.5.5]{Tho03}. The implication $(i)\Rightarrow(ii)$ allows us to deduce invariance properties for $\mathcal U$ even when $P$ is not explicitly known, e.g., when $\mathcal T_t$ is known to preserve convexity for a.e.\ $t$ \cite{BatBob12}.




For our purposes, a particularly interesting instance of closed convex sets are order intervals in Hilbert lattices: we hence assume in the following $H$ to be a Hilbert lattice. It is known that each separable Hilbert lattice is isometrically lattice isomorphic to a Lebesgue space $L^2(X)$ for some $\sigma$-finite measure space $(X,\Sigma,\mu)$, see e.g.~\cite[Cor.~2.7.5]{Mey91}. Accordingly, we can consider the set $H_{\R}:=L^2(X;\R)$ of real-valued functions. Let $a,b\in \mathbb R\cup\{\pm\infty\}$: we introduce the (bounded or unbounded) \emph{order intervals}
\[
{[a,b]_H}:=\left\{f \in H_\R: a\le f(x)\le b\hbox{ for a.e.\ }x\in X\right\}:
\]
which  are closed convex subsets of $H$. Many qualitative properties of solutions to evolution equations can be described by means of invariance of order intervals under the flow that governs the associated Cauchy problems.


\begin{definition}\label{submarkovian EVF} Let $(X,\Sigma,\mu)$ be a $\sigma$-finite measure space. An evolution family $\U$ on the Hilbert lattice $L^2(X)$ is called
	\begin{enumerate}[(a)]
\item \emph{real} if $U(t,s)H_\R\subset H_\R$ for all $(t,s)\in \Delta$; 
\item \emph{positive} if it is real and $U(t,s)[0,\infty[_H\subset [0,\infty[_H$ for all $(t,s)\in \Delta$;
\item \emph{$L^p$-contractive}, $p\in [1,\infty]$, if $U(t,s)$ maps $\{f\in L^2(X)\cap L^p(X):\|f\|_{L^p}\le 1\}$ into itself for all $(t,s)\in \Delta$;
	\item \emph{completely contractive} if it is both $L^1$-contractive and $L^\infty$-contractive;
\item \emph{completely quasi-contractive} if there is some constant $\tilde{\omega}$ such that the rescaled evolution family $\U_{\tilde \omega}$ defined by
	\begin{equation}\label{eq:rescaled-tilde}
	U_{\tilde\omega}(t,s):=e^{-\tilde\omega(t-s)}U(t,s),\qquad (t,s)\in \Delta,
	\end{equation}
 is completely contractive;
		\item \emph{sub-Markovian} if it is positive and $L^\infty$-contractive; \emph{Markovian} if additionally $\|U(t,s)\|_{{\mathcal L}(L^\infty)}=1$;
{\item \emph{sub-stochastic} if it is positive and $L^1$-contractive}; \emph{stochastic} if additionally and $\|U(t,s)f\|_{L^1(X)}=\|f\|_{L^1(X)}$ for all $0\leq f\in L^2(X)\cap L^1(X)$ and all $(t,s)\in\Delta$.
	\end{enumerate}
\end{definition}

\begin{remark} 

 Let $\U$ be a completely quasi-contractive evolution family on $L^2(X).$ Then by the Riesz--Thorin Theorem the rescaled evolution family $\U_{\tilde{\omega}}$ is $L^p$-contractive for all $p\in[1,\infty]$ and each $U(t,s)$ can be extended from $L^p(X)\cap L^2(X)$ to a quasi-contractive operator $U_p(t,s)$ on $L^p(X)$ for all $p\in [1,\infty].$ The extrapolated family $\U_p:=\{U_p(t,s) \mid (t,s)\in\Delta\}$ is consistent, i.e.,	for all $p\in [1,\infty]$
	\[U_p(t,s)f=U_2(t,s)f\qquad \hbox{for all }(t,s)\in \Delta \hbox{ and all }f\in L^p(X)\cap L^2(X).
	\]
Clearly, $U_p(s,s)=I_{H}$ and $U_p(t,s)=U_p(t,r)U_p(r,s)$ for all $0\leq s\leq r\leq t\leq T$ and $p\in[1,\infty].$ Moreover, by the interpolation inequality (Hölder inequality) we obtain that $\U_p$ is strongly continuous on $\overline{\Delta}$  for all $p\in ]1,\infty[$. Using a similar argument as in \cite[Prop.\ 4]{Voi92} we conclude that $\U_p$ is a strongly continuous evolution family on $L^p(X)$ for all $p\in [1,\infty[$.

\end{remark}


For future reference let us note explicitly the following consequence of Proposition~\ref{cor:easytho}.
\begin{proposition}\label{proposition Linftycontractivity} 
The evolution family $\U$ associated with $\fra\in\Formm([0,T];V,H)$ is
\begin{enumerate}[(1)]
\item positive provided $(\Re v)^+\in V$, $\Re a(t;\Re v,\Im v)\in \mathbb R$, and $\Re a(t;(\Re v)^+,(\Re v)^-)\le 0$ for all $v\in V$ and a.e.\ $t\in[0,T]$.
\item $L^\infty$-contractive, provided $(1\wedge |v|)\sgn v\in V$ and $\Re a(t;(1\wedge |v|)\sgn v,(|v|-1)^+\sgn v)\geq 0$ for all $v\in V$ and a.e.\ $t\in[0,T]$.
\end{enumerate}
\end{proposition}



Let us state a further consequence of Proposition~\ref{cor:easytho} concerning irreducibility of evolution families on $L^2(X)$ on a given $\sigma$-finite measure space $(X,\Sigma,\mu)$. We denote by ${\bf 1}_{\Xi}$ the characteristic function of any given $\Xi\in\Sigma$.

We can now provide sufficient conditions for the evolution family to converge towards a rank-one projector, thus extending to the non-autonomous setting one of the main results of the classical theory of positive semigroups. Given $\fra \in \Formm([0,T];V,H)$, if  
\begin{equation}\label{irred-1}
(\Re v)^+\in V,\ \Re a(t;\Re v,\Im v)\in \mathbb R\hbox{ and }\Re a(t;(\Re v)^+,(\Re v)^-)\le 0\quad\hbox{for all }v\in V\hbox{ and a.e.\ }t\in[0,T]
\end{equation}
and
\begin{equation}\label{irred-2}
\hbox{for all }\Xi\in\Sigma\quad {\bf 1}_\Xi V\subset V \quad\hbox{implies}\quad \mu(\Xi)=0\hbox{ or }\mu(X\setminus\Xi)=0,
\end{equation}
then the holomorphic semigroups $\mathcal T_t$ are for a.e.\ $t\in ]0,T[$ positivity improving; if they are additionally compact, then by classical Perron--Frobenius theory there is a spectral gap of size $\tilde{s}(t)>0$ between their dominant eigenvalue and the bottom 
\[
\inf\left\{\Re \lambda>0: \lambda\in\sigma(A(t))\right\}
\]
of the remaining spectrum of their generators $A(t)$.

\begin{proposition}\label{cor:stability}
Let $\fra \in\Formm([0,\infty[;V,H)$, with $V$ compactly embedded in $H$. Assume that~\eqref{irred-1} and~\eqref{irred-2} hold and that furthermore $\tilde{s}\in L^1_{loc}(]t_0,\infty[)$ and $\liminf\limits_{t\to\infty}\frac{1}{t-t_0}\int_{t_0}^t \tilde{s}(r)\ dr=\tilde{\Sigma}>0$ for some $t_0>0$.
If for a.e.\ $t\in [0,\infty[$
\begin{itemize}
\item the spectral bound of $A(t)$ is 0,
\item $\Ker(A(t))$ is spanned by the same  vector $u$, and
\item $\Ker(A^*(t))$ is spanned by the same  vector $\phi$, with $(\phi|u)=1$,
\end{itemize}
then
\[
\lim_{t\to\infty} U(t,s)=P:=\phi\otimes u\qquad\hbox{in norm  for all $s\ge 0$.}
\]
\end{proposition}
Under the assumptions of Proposition~\ref{cor:stability}, the vectors $u,\phi$ are strictly positive.
\begin{proof}
The part $\tilde{A}(t)$ of $A(t)$ in the Hilbert space $\Ker(P)$ is associated with a form $\tilde{\fra}\in \Formm(]0,\infty[;V\cap \Ker(P),\Ker(P))$ that satisfies 
\[
\tilde{a}(t;u,u)\ge \tilde{s}(t)\|u\|^2\qquad\hbox{for a.e. }t\in ]0,\infty[ \hbox{ and all }u\in V\cap \Ker(P).
\]
Given a partition $\Lambda=(\lambda_0,\ldots,\lambda_{n})$ of a compact interval $[a,b]$, we can hence define in the usual way the averaged forms $\tilde{a}_k$, which satisfy
\[
\Re \tilde{a}_k(u,u)\geq \tilde{s}_k \|u\|^2,\qquad u\in V\cap \Ker(P),
\]
for
\[
\tilde{s}_k:=\frac{1}{\lambda_{k+1}-\lambda_k} \int_{\lambda_{k}}^{\lambda_{k+1}}\tilde{s}(r)\d r.
\]
Therefore, the associated semigroup $(e^{-r\tilde{A}_k})_{r\ge 0}$ is uniformly exponentially stable: more precisely
\[
\|e^{-r\tilde{A}_k}\|=\|e^{-r{A}_k}(I-P)\|\le e^{-\tilde{s}_k r}
\qquad \hbox{for all }r\ge 0.
\]
Observe now that
\[
\begin{split}
U_\Lambda (b,a)-P&=U_\Lambda (b,a)P-P+U_\Lambda (b,a)(I-P)\\
&=e^{-(\lambda_{n}-\lambda_{n-1}) A_{n}}\cdots e^{-(\lambda_1-\lambda_0) A_1}P-P+U_\Lambda(b,a)(I-P)^n\\
&=U_\Lambda(b,a)(I-P)^n.
\end{split}
\]
Accordingly,
\[
\begin{split}
\|U_\Lambda (b,a)-P\|&\le \|e^{-(\lambda_{n}-\lambda_{n-1}) A_{n}}(I-P)\|\cdots
\|e^{-(\lambda_1-\lambda_0) A_1}(I-P)\|\\
&\le e^{-(\lambda_{n+1}-\lambda_n) \tilde{s}_n}\cdots e^{-(\lambda_1-\lambda_0) \tilde{s}_1}=e^{-\int_{a}^b \tilde{s}(r)\d r}
\end{split}
\]
whence
\[
\lim_{t\to\infty}\|U_\Lambda (t,t_0)-P\|\le \lim_{t\to\infty}e^{-\int_{t_0}^t \tilde{s}(r)\d r}\le \lim_{t\to\infty}e^{-\tilde{\Sigma}(t-t_0)}
\]
and for all $x\in H$ by Theorem \ref{convergence forte de solution approchee}
\[
\lim_{t\to\infty}\|U(t,t_0)x-Px\|=\lim_{t\to\infty}\lim_{|\Lambda|\to 0} \|U_\Lambda (t,t_0)x-Px\|\le \lim_{t\to\infty}e^{-\tilde{\Sigma}(t-t_0)}\|x\|.
\]
This concludes the proof.
\end{proof}

In the following sections we will often need to discuss 
complete contractivity. In order to find sufficient conditions, observe that $\U$ is $L^1$-contractive if and only if $U(t,s)^*$ is $L^\infty$-contractive for all $(t,s)\in \Delta$. How to prove $L^\infty$-contractivity of all $U(t,s)^*$?
Consider the non-autonomous adjoint form $\fra^*:[0,T]\times V\times V\to \C$ of $\fra$ defined by $\fra^*(t;u,v):=\overline{a(t;v,u)}$ for all $t\in[0,T]$ and $u,v\in V$. While $\fra^*\in \Formm([0,T];V,H)$, too, and hence $\fra^*$ is associated with an evolution family $(U_*(t,s))_{(t,s)\in \overline{\Delta}}$, one has in general $U_*(t,s) \neq U(t,s)^*$.
 However, it was observed in~\cite[Thm.~2.6]{Dan00} that the \emph{returned adjoint form} $\cev{\fra^*}:[0,T]\times V\times V\to \C$ of $\fra$ defined by \[
\cev{\fra^*}(t;u,v):=\fra^*(T-t;v,u),\quad t\in [0,T],\ u,v\in V,
\]
which clearly belongs to $\Formm([0,T];V,H)$, too, is associated with an evolution family $\cev{\U^*}$ that satisfies
	\begin{equation}\label{eq:retadjform}
	\big [ \cev{U^*}(t,s)\big ]^* f=U(T-s,T-t)f\qquad \hbox{for all }f\in H\hbox{ and }(t,s)\in{\Delta}.
	\end{equation}
In particular, $\U$ is $L^1$-contractive if and only if $\cev{\U^*}$ is $L^\infty$-contractive; $\U$ is completely contractive if so is $\cev{\U^*}$; and by Proposition \ref{proposition Linftycontractivity} we conclude the following.

\begin{proposition}\label{prop:L1contractivity} 
The evolution family $\U$ associated with $\fra\in\Formm([0,T];V,H)$ is
\begin{enumerate}[(1)]
\item $L^1$-contractive provided $(1\wedge |u|)\sgn u\in V$ and $\Re a(t;(|v|-1)^+\sgn v,(1\wedge |v|)\sgn v)\geq 0$ for all $v\in V$ and a.e.\ $t\in[0,T]$;
\item completely contractive provided $(1\wedge |u|)\sgn u\in V$ and $\Re a(t;(|v|-1)^+\sgn v,(1\wedge |v|)\sgn v)\geq 0$, $\Re a(t;(1\wedge |v|)\sgn v,(|v|-1)^+\sgn v)\geq 0$ for all $v\in V$ and a.e.\ $t\in[0,T]$.
\end{enumerate} 
\end{proposition}

We can now give a sufficient condition for $L^p$-quasi-contractivity of the evolution family $\U$ that governs the Cauchy problem \eqref{evolution equation u(s)=x}. 


\begin{theorem}\label{theorem quasilpcontractivity} Let $\fra\in \Formm([0,T]; V,H)$ and let $p\in]1,\infty[$ be given. Assume that there exists a function $\hat\omega_p\in L^\infty(0,T)$ such that $(\mathcal T_t)_{t\in [0,T]}$ satisfies
\begin{equation}
\label{Eq1 them: theorem quasilpcontractivity}
\|e^{-r A(t)}f\|_{L^p(X)}\leq e^{r\hat\omega_p(t)}\|f\|_{L^p(X)}\qquad \hbox{for all }f\in L^2(X)\cap L^p(X),\ r\geq 0,\hbox{ and a.e. }t\in [0,T].
\end{equation}
Then the evolution family $\mathcal U$ associated with $\fra$ extrapolates to a consistent evolution family on $L^p(X)$ and 
 \begin{equation}
	\label{eq them:quasilpcontractivity EVF}
	\|U(t,s)f\|_{L^p(X)}\leq e^{\int_s^t\hat\omega_p(r)dr}\|f\|_{L^p(X)}\qquad \hbox{for all }f\in L^2(X)\cap L^p(X) \hbox{ and }(t,s)\in \Delta.
	\end{equation}
\end{theorem}

\begin{proof} The case where $\hat{\omega}_p\equiv 0$ follows directly from Proposition~\ref{cor:easytho}; the general case is slightly more delicate. First, applying \cite[Thm.~4.1]{Nit12} to the non-autonomous form $a(t,\cdot,\cdot)+\hat\omega_p(t)(\cdot\mid \cdot)_{L^2}$, we see that the assumption~\eqref{Eq1 them: theorem quasilpcontractivity} is equivalent to the following condition: $P_{B^p}V\subset V$ and for a.e.\ $t\in[0,T]$
	\begin{equation}\label{Eq2 them: theorem quasilpcontractivity} 
	\Re a(t;u,|u|^{p-2}u)+\hat\omega_p(t)\|u\|^p_{L^p}\geq 0\quad \hbox{ for all }u\in V\hbox{ s.t. }|u|^{p-2}u\in V.
	\end{equation}
	(Here $B^p$ denotes the $L^p$-unit ball and  $P_{B^p}$ is the projector of $L^2(X)$ onto $B^p$.)
	
Let now $[s,t]\subset [0,T]$ and let $\Lambda=(\lambda_0,\lambda_1,\ldots,\lambda_n)$ be a partition of $[s,t]$ and let $\fra_k: V\times V\to\C$, $k=0,1,\ldots,n$, be the family of bounded $H$-elliptic forms given by~\eqref{eq:form-moyen integrale} and $(e^{-rA_k})_{r\geq 0}$ be the associated $C_0$-semigroup. Furthermore, define the finite real sequence $\hat\omega_{k,p}$, $k=0,1,\ldots, n$, as follows 
\begin{equation}\label{eq:omega-moyen integrale-p}
\hat\omega_{k,p}:=\frac{1}{\lambda_{k+1}-\lambda_k}
\int_{\lambda_{k}}^{\lambda_{k+1}}\hat\omega_p(r) \d r \qquad k=0,1,\ldots,n.
\end{equation}
Then~\eqref{Eq2 them: theorem quasilpcontractivity} implies that for all $k=0,1,\ldots,n$	
	\begin{equation}\label{Eq3 them: theorem quasilpcontractivity} \Re a_k(u,|u|^{p-2}u)+\hat\omega_{k,p}\|u\|_{L^p}^p\geq 0,\quad\hbox{	for all }u\in V\hbox{ s.t. }|u|^{p-2}u\in V.
	\end{equation}
Again applying~\cite[Thm.~4.1]{Nit12} to the form $a_k+\hat\omega_{k,p}(\cdot\mid \cdot)_{L^2}$ we obtain
	\begin{equation}
	\label{Eq4 them: theorem quasilpcontractivity}
	\|e^{-rA_k}\|_{L^p(X)}\leq e^{r\hat\omega_{k,p}} \quad \hbox{for all }s\geq 0\hbox{ and } k=0,1,\ldots,n.
	\end{equation}
	Thus, using~\eqref{promenade1}-\eqref{promenade2} we find
	\begin{equation}
	\label{Eq5 them: theorem quasilpcontractivity}
	\|U_\Lambda(t,s)\|_{L^p(X)}\leq e^{\int_s^t\hat\omega_p(r)dr}\qquad \hbox{for all }(t,s)\in \Delta\hbox{ and each partition $\Lambda$ of }[s,t].
	\end{equation}
	Thus the desired estimate~\eqref{eq them:quasilpcontractivity EVF} follows from Theorem~\ref{convergence forte de solution approchee} and Fatou's Lemma.
\end{proof}

In a similar way we can discuss stochasticity, another feature that cannot be easily interpreted as an invariance property.
\begin{proposition} \label{them: stochastic EVF} The evolution family $\U$ associated with $\fra\in\Formm([0,T];V,H)$ is stochastic provided $(\Re u)^+\in V$, $\Re a(t,\Re u,\Im u)\in \mathbb R$, $\Re a(t;(\Re u)^+,(\Re u)^-)\le 0$, ${\bf 1}\in V$, and $a(t;{\Re u},{\bf 1})=0$ for all $v\in V$ and a.e.\ $t\in[0,T]$.
\end{proposition}


Our last result in this section is devoted to the issue of domination of evolution families.

\begin{proposition}\label{prop:domin}
Let $\fra\in \Formm([0,T];V,H)$ and denote as usual by $\U$ 
the associated evolution family.
Let furthermore $W$ be a separable Hilbert space that is densely and continuously embedded in $H$ and ${\mathfrak b}\in \Formm([0,T];W,H)$: we denote by $\V$
the associated evolution family.
Assume that 
\begin{itemize}
\item $\Re u\in V$ and $(\Re u)^+\in V$ for all $u\in V$;
\item $V$ is a \textit{generalized ideal} of $W$, i.e.,
\begin{itemize}
\item $u\in W$ implies $|u|\in W$ and
\item $u_1\in V$ and $u_2\in W$ are such that $|u_2|\le |u_1|$, then $u_2 \sgn u_1\in V$;
\end{itemize}
\item $\Re a(t;\Re u,\Im u)\in \R$ for all $u\in V$;
\item $\Re a(t;(\Re u)^+,(\Re u)^-)\le 0$ for all $u\in V$;
\item $\Re a(t;u,v)\ge b(t;|u|,|v|)$ for all $u,v\in V$ s.t.\ $u\overline{v}\ge 0$.
\end{itemize}
Then $\U$ is dominated by $\V,$ i.e.,
\begin{equation}
\label{Def: dominatioon EVF}
|U(t,s)f|\leq V(t,s)|f|\qquad \hbox{for all }(t,s)\in \Delta \hbox{ and }f\in H.
\end{equation} 

\end{proposition}

\begin{proof} Let $\Lambda$ be a partition of $[0,T].$ Define the piecewise constant $\frb_\Lambda\in \Formm([0,T]; W,H)$ via formulae which are analogous to~\eqref{eq:form-moyen integrale} and~\eqref{form: approximation formula2} and let $\V_\Lambda$ be the associated evolution family. 
By~\cite[Thm.~4.1]{ManVogVoi05} we see that the semigroup associated with the averaged forms $\frb_\Lambda$ dominate $\mathcal T_k$, hence that $\V_\Lambda$ dominates $\U_\Lambda$.
By Theorem \ref{convergence forte de solution approchee} we conclude that $\U$ is dominated by $\V.$ 
\end{proof}

\section{Ultracontractivity}\label{sec:ultra}

In this and the next section we are going to restrict to the case of $H=L^2(X)$, where $X$ an $\sigma$-finite measurable space. Recall that a $C_0$-semigroup $\mathcal S$ on $L^2(X)$ is said to be \emph{ultracontractive} if there exist constants $c_0,n>0$, and $\tilde\omega\in\mathbb R$ such that 
\begin{equation}
\label{Def: ultracontractivity semigroup}
\|S(r)f\|_{L^\infty(X)}\leq c_0r^{-\frac{n}{2}} e^{r\tilde\omega}\|f\|_{L^1(X)}\qquad \hbox{ for all } r\geq 0\hbox{ and all }f\in L^2(X)\cap L^1(X).
\end{equation}

In this section we are going to develop a theory of ultracontractive evolution families.


\begin{definition}\label{def:ultracont}
	We call an evolution family $\U$ on $L^2(X)$ \emph{ultracontractive} if there exist constants $c_0,n>0$ and $\tilde\omega\in\mathbb R$ such that
	\begin{equation}\label{eq: L1linfty}
	\|U(t,s)f\|_{L^\infty(X)}\leq c_0(t-s)^{-\frac{n}{2}}
	 e^{(t-s)\tilde\omega}\|f\|_{L^1(X)}
	\qquad \hbox{ for all }(t,s)\in\Delta\hbox{ and }f\in L^2(X)\cap L^1(X).
	\end{equation}
\end{definition}

By a direct consequence of the Kantorovitch--Vulikh Theorem, see e.g. \cite[Theorem 1.3]{AreBuk94}, any ultracontractive evolution family $\U$ is given by an integral kernel: more precisely, there exists a family $(\Gamma(t,s))_{(t,s)\in \Delta}\subset L^\infty(\Omega\times \Omega)$ such that
\[
U(t,s)f(x)=\int_{\Omega} \Gamma(t,s;x,y)f(y)\ dy\qquad \hbox{for all }(t,s)\in \Delta, \ f\in L^2(X)\cap L^1(X)\hbox{ and a.e.\ }x\in \Omega,
\]
with
\[
\|U(t,s)\|_{\L(L^1(X),L^\infty(X))}=\|\Gamma(t,s)\|_{L^\infty(\Omega\times\Omega)},\qquad \hbox{for all } (t,s)\in \Delta.
\]

It is well-known that
ultracontractivity of semigroups can be deduced from the Nash or Gagliardo--Nirenberg inequalities for the domain of the associated form, see~\cite[Chapt.~6]{Ouh05}. We are going to extend this result to the non-autonomous setting. For that we need the following definitions.

\begin{definition}
Let $V$ be a subspace of $L^2(X)$. The space $V$ is said to satisfy 

(i) a \emph{Nash inequality} if there exist constants $C_N,\mu>0$ such that 
\begin{equation}
\label{Nash inequality} \|u\|_{L^2(X)}^{2+\frac{4}{\mu}}\leq C_N\|u\|_V^2\|u\|_{L^1(X)}^{\frac{4}{\mu}}\qquad \hbox{for all }u\in L^1(X)\cap V;
\end{equation}

(ii) a \emph{Gagliardo--Nirenberg inequality} 
if there exist constants $C_G, N>0$ such that 
\begin{equation}\label{Gagliardo--Nirenberg inequality} 
\|u\|_{L^q(X)}\leq C_G \|u\|_{L^2(X)}^{1-N\frac{q-2}{2q}}\|u\|_{V}^{N\frac{q-2}{2q}}\quad\hbox{for all }u\in V
\end{equation}
holds for all $q\in ]2,\infty[$ such that $N\frac{q-2}{2q}\leq 1$.
\end{definition}
Sobolev spaces $H^1(I)$ on intervals $I\subset \R$ satisfy e.g.\ the Nash inequality, see e.g.~\cite[\S~1.4.8]{Maz85}. More generally, the same is true for each closed subspace $V$ of $H^1(X)$ which has the $L^1-H^1$-extension property \cite[Lemma 2.7]{AreEls97}, where $\Omega$ is an arbitrary open set of $\R^d$. Several geometric conditions on $\Omega\subset \R^d$ under which a Sobolev space $V=H^k(\Omega)$ satisfies a Gagliardo--Nirenberg inequality are known, see e.g.~\cite[Chapter 5]{AdaFou03}.

Here and in the following, we are adopting the usual notations introduced in~\eqref{measurability}--\eqref{ellipticity}.
\begin{theorem}\label{Theorem: Ultracontractivity via Nash ineq} 
Let $\fra\in \Formm([0,T];V,H)$ such that the associated evolution family $\U$ is completely quasi-contractive with constant $\tilde\omega\in\R$. If $V$ satisfies a Nash inequality \eqref{Nash inequality} for some constants $\mu,C_N>0$, then $\U$ is ultracontractive and 
	\begin{equation}\label{equ: Ultracontractivity via Nash ineq}
\|U(t,s)\|_{\L(L^1(\Omega),\L^\infty(\Omega))}\leq \Big(\frac{\mu C_N}{4\alpha}\Big)^{\frac{\mu}{2}}(t-s)^{-\frac{\mu}{2}}e^{\max\{\omega,\tilde\omega\}(t-s)}\quad\hbox{for	all }(t,s)\in\Delta.
\end{equation}
\end{theorem}

\begin{proof}The first part of the proof is similar to that of \cite[Prop.~3.8]{AreEls97}. Upon rescaling $U(t,s)$ by $e^{-\max\{\omega,\tilde{\omega}\}(t-s)}$ we can without loss of generality assume both $\fra$ to be coercive and the evolution family $\U$ to be completely contractive.  Let $f\in L^1(X)\cap V$ and let $s\in[0,T[$ be fixed. If $y\in \MaxR (V,V')$ then
$\|y(\cdot)\|^2_{H}\in W^{1,1}(s,T;V')$ and
\begin{equation}\label{poduct rule classique}
\frac{\d}{\d t}\Vert y(\cdot)\Vert^2_H=2\Re\langle\dot y(\cdot),y(\cdot)\rangle
\end{equation}
by~\cite[Prop.~III.1.2]{Sho97}: accordingly, using ~\eqref{ellipticity} and since $t\mapsto U(t,s)f\in MR(s,T;V,V')$ we obtain that for all $f\in V\cap L^1(X)$ and a.e.\ $(t,s)\in \Delta$
\begin{align*}
\frac{\partial }{\partial t}\|U(t,s)f\|_{L^2(X)}^2&=2\Re \langle\frac{\partial }{\partial t} U(t,s)f, U(t,s)f\rangle
\\&=-2\Re \langle \A(t)U(t,s)f,U(t,s)f \rangle
\\&=-2a(t;U(t,s)f,U(t,s)f)\\
&\leq -2\alpha \|U(t,s)f\|_{V}^2 \\
&\leq -\frac{2\alpha}{C_N}\|U(t,s)f\|_{L^2(X)}^{2+\frac{4}{\mu}}\|U(t,s)f\|_{L^1(X)}^{-\frac{4}{\mu}}.
\end{align*} 
It follows that
\begin{align*}
\frac{\partial }{\partial t}\Big(\|U(t,s)f\|_{L^2(X)}^2\Big)^{-\frac{2}{\mu}}&=
-\frac{2}{\mu}\|U(t,s)f\|_{L^2(X)}^{-2-\frac{4}{\mu}} \frac{\partial }{\partial t}\| U(t,s)f\|_{L^2(X)}^2\ge\frac{4\alpha}{\mu C_N} \|U(t,s)f\|_{L^1(X)}^{-\frac{4}{\mu}}	
\\&\geq \frac{4\alpha}{\mu C_N} \|f\|_{L^1(X)}^{-\frac{4}{\mu}}
\end{align*}
since $\U$ is completely contractive. Integrating this inequality between $s$ and $t$ we find 
\begin{equation}\label{eq: proof Thm 1-infty}
\|U(t,s)\|_{\L(L^1(X),L^2(X))}\leq \Big(\frac{\mu C_N}{4\alpha}\Big)^\frac{\mu}{4}(t-s)^{-\frac{\mu}{4}}	\quad \hbox{ for all }(t,s)\in\Delta.
\end{equation}
In order to obtain the $L^2-L^\infty$-bound and thus prove the claimed ultracontractivity we will use the returned adjoint form $\cev{\fra^*}$ introduced in Section~\ref{sec:lattice}. In fact, arguing as in the first part of the proof we find that the evolution family $\cev{U^*}$ associated with $\cev{\fra^*}$ satisfies~\eqref{eq: proof Thm 1-infty} with the same bound. Then using the identity~\eqref{eq:retadjform} we conclude that 
	\begin{equation}
	\|U(t,s)\|_{\L(L^2(X),\L^\infty(X))}\leq \Big(\frac{\mu C_N}{4\alpha}\Big)^{\frac{\mu}{4}}(t-s)^{-\frac{\mu}{4}}\quad\hbox{for	all }(t,s)\in\Delta.
	\end{equation}
Finally, the evolution law satisfied by $\U$ completes the proof.
\end{proof}

	



%
\begin{definition}\label{def:linearlylpc-new}
An evolution family $\U$ on $L^2(X)$ is called \emph{linearly quasi-contractive} if for some constants $\alpha_1,\alpha_2$ independent of $p$
	\begin{equation}
	\label{eq them:quasilpcontractivity C0S-2}
	\|U(t,s)f\|_{L^p(X)}\leq e^{(t-s)(\alpha_1+p\alpha_2)}\|f\|_{L^p(X)}
	\end{equation}
for all $(t,s)\in\Delta,\ f\in L^2(X)\cap L^p(X), \hbox{ and } p\in [2,\infty[.$

\end{definition}

Linear $L^p$-quasi-contractivity turns out to be a key notion when it comes to checking ultracontractivity when the domain of the form satisfies a Gagliardo--Nirenberg inequality. Indeed, in this case we have the following result.
\begin{theorem}\label{Thm2: Ultracontractivity} 
Let $\fra\in \Formm([0,T];V,H)$.
Assume that $\U$ and $\cev{\U^*}$ are both linearly quasi-contractive with constants $\alpha_1,\alpha_2,\alpha_1^\ast,\alpha_2^\ast$. If $V$ satisfies a Gagliardo--Nirenberg inequality for some $C_G,N>0$, then $\U$ is ultracontractive and we have 
	\begin{equation}\label{2Ultracontractivity for closed form}
	\|U(t,s)\|_{\L(L^{1}(X),L^{\infty}(X))}\leq C_G^{\frac{N}{2}}\alpha^{-\frac{N^2}{4(N+2)}}e^{\bar\omega(t-s)}(t-s)^{-\frac{N^2}{2(N+2)}}\quad \hbox{for	all }(t,s)\in\Delta,
	\end{equation}
	where 
\[
\bar\omega:=\omega+\alpha_1+\alpha_1^*+\frac{2(N+2)}{N}[\alpha_2+\alpha_2^*].
\]
\end{theorem}

 In the proof of Theorem~\ref{Thm2: Ultracontractivity} we will need the following lemma.


\begin{lemma}\label{Thm1: Ultracontractivity} Let $\fra\in \Formm([0,T];V,H)$ and let $\U$ and $\cev{\U^*}$ be the evolution families associated with $\fra$ and $\cev{\fra^*}$, respectively. Assume that $\U$ and $\cev{\U^*}$ are both linearly quasi-contractive (with constants $\alpha_1,\alpha_2,\alpha_1^\ast,\alpha_2^\ast$). In addition we assume that there exist constants ${\kappa_1},\kappa>0$ such that
	\begin{equation}\label{eq: l2l2*}
	\|U(t,s)\|_{\L(L^2(X),L^{\hat 2}(X))}\leq { \bf \kappa}(t-s)^{-\frac{\kappa_1}{2}} { e^{(t-s)[\alpha_1+\hat 2\alpha_2]}}\qquad \hbox{for all }(t,s)\in \Delta
	\end{equation} 
	and 
	\begin{equation}\label{eq: l2l2*Returned adjoint}
	\|\cev{U^*}(t,s)\|_{\L(L^2(X),L^{\hat 2 }(X))}\leq { \kappa}(t-s)^{-\frac{\kappa_1}{2}}{e^{(t-s)[\alpha_1^*+\hat 2\alpha_2^*]} }\qquad \hbox{for all }(t,s)\in \Delta
	\end{equation} 
where $\hat 2 :=\frac{2N}{N-2}$ for some integer $N\ge 3$. Then $\U$ is ultracontractive: more precisely,
\begin{equation}\label{psedoUltracontractivity}
	\|U(t,s)\|_{\L(L^{1}(X),L^{\infty}(X))}\leq\kappa^{\frac{N}{2}}c(t-s)^{-\frac{N\kappa_1}{2}}e^{\tilde\omega(t-s)}\qquad\hbox{ for	all }(t,s)\in\Delta\hbox{ with }t> s
\end{equation}
for some positive constants $c,\mu>0$ that depend only on $N, \kappa_1$, where
\[
\tilde\omega:=[\alpha_1+\alpha_{1}^*]+\mu[\alpha_2+\alpha_{2}^*].
\] 
\end{lemma}
\begin{proof} For the proof we follow similar argument as in \cite[Thm.~5.2]{Ouh04} and \cite{Cou93} where ultracontractivity for semigroups are treated. 

\textit{Step 1.} We will first prove that 
	\begin{equation}\label{eq: l2linfty}
	\|U(t,s)\|_{\L(L^{2}(X),L^{\infty}(X))}\leq \kappa^{\frac{N}{2}} C(t-s)^{-\frac{\kappa_1N}{4}}e^{\alpha_1(t-s)}e^{\mu\alpha_2(t-s)},
	\end{equation}
	where the positive constants $C$ and $\mu$ depend only on $N$ and $\kappa_1.$
For some $r>2$ that will be fixed later we can combine~\eqref{eq: l2l2*} with the linear $L^r$-quasi-contractivity of $\U$ and obtain by a version of Riesz--Thorin interpolation theorem \cite[Thm.~2.2.14]{Dav07} that for any $\theta\in[0,1]$
	\begin{equation*}\label{eq2: l2l2*}
	\|U(t,s)\|_{\L(L^{p_1}(X),L^{q_1}(X))}\leq { \bf \kappa^\theta}(t-s)^{-\frac{\kappa_1\theta}{2}}e^{(1-\theta)(t-s)[\alpha_1+r\alpha_2]}e^ {\theta(t-s)[\alpha_1+\hat 2\alpha_2]} \quad
	\end{equation*} 
where $\frac{1}{p_1}:=\frac{1-\theta}{r}+\frac{\theta}{2}, \frac{1}{q_1}:=\frac{1-\theta}{r}+\frac{\theta}{\hat 2}$. Let now $p\in]2,\infty[.$ Choosing $\theta:=\frac{1}{p}$ and $r=2(p-1)$ in the above equation we obtain that
	\begin{equation}\label{eq: lplp*}
	\|U(t,s)\|_{\L(L^{p}(X),L^{N_p}(X))}\leq { \bf \kappa}^\frac{1}{p}(t-s)^{-\frac{\kappa_1}{2p}}e^{(1-\frac{1}{p})(t-s)[\alpha_1+2(p-1)\alpha_2]}e^{\frac{1}{p}(t-s)[\alpha_1+\hat 2\alpha_2]}
	\end{equation} 
	holds for all $p\in ]2,\infty[$ where $N_p:=p\frac{N}{N-1}.$

Next, set $R=\frac{N}{N-1},p_k=2R^k$ and $t_k=\frac{N+1}{2N} (2R)^{-k}=\frac{N+1}{Np_k} 2^{-k}$ for all $k\in \N.$ Moreover, let $s_0=s$ and $s_{k+1}=s_k+t_k(t-s)$ for each integer $k>0.$ Then we have $\sum_{k} t_k=1, \sum_{k} \frac{1}{p_k}=\frac{N}{2}.$ Furthermore, $s_{k+1}<s_k$ for all $k\in \N$ and $t=\lim\limits_{k\to \infty}s_k.$ Thus, applying~\eqref{eq: lplp*} for $p=p_k,$ using \eqref{eq them:quasilpcontractivity C0S-2} and the evolution law satisfied by $\U$ we deduce that 
	\begin{align*}
	\|U(t,s)\|_{\L(L^{2}(X),L^{\infty}(X))}&\leq \prod_{k\geq 0}\|U(s_{k+1},s_k)\|_{\L(L^{p_k}(X),L^{N_{p_k}}(X))}
	\\&\leq \prod_{k\geq 0} \kappa^{\frac{1}{p_k}}t_k^{-\frac{\kappa_1}{2p_k}}(t-s)^{-\frac{\kappa_1}{2p_k}}e^{(1-\frac{1}{p_k})(s_{k+1}-s_k)[\alpha_1+2(p_k-1)\alpha_2]+\frac{1}{p_k}(s_{_{k+1}}-s_k)[\alpha_1+\hat 2\alpha_2]}
	\\&= \kappa^{\frac{N}{2}}(t-s)^{-\frac{\kappa_1 N}{4}}\prod_{k\geq 0} t_k^{-\frac{\kappa_1}{2p_k}}e^{(1-\frac{1}{p_k})(s_{k+1}-s_k)[\alpha_1+2(p_k-1)\alpha_2]+\frac{1}{p_k}(s_{_{k+1}}-s_k)[\alpha_1+\hat 2\alpha_2]}	\\
&=\kappa^{\frac{N}{2}}(t-s)^{-\frac{N\kappa_1}{4}}\prod_{k\geq 0} t_k^{-\frac{\kappa_1}{2p_k}}\prod_{k\geq 0} e^{(t-s)t_k\alpha_1}\prod_{k\geq 0}e^{(t-s)t_k\alpha_2[p_k-1-\frac{p_k-1}{p_k}+ \frac{\hat 2}{p_k}]}
\\&=\kappa^{\frac{N}{2}}(t-s)^{-\frac{\kappa_1N}{4}}\prod_{k\geq 0} t_k^{-\frac{\kappa_1}{2p_k}}e^{\alpha_1(t-s)}e^{\alpha_2(t-s)\sum_{j\geq 0}t_j\Big[\frac{(p_j-1)^2+\frac{2N}{N-2}}{p_j}\Big]}
\\&{\leq} \kappa^{\frac{N}{2}} C(t-s)^{-\frac{\kappa_1N}{4}}e^{\alpha_1(t-s)}e^{\mu\alpha_2(t-s)},
	\end{align*}
where the positive constants $C$ and $\mu$ depend only on $N$ and $\kappa_1.$
	
\textit{Step 2.} It remains to estimate $U(t,s)$ in $\L(L^{1}(X),L^{2}(X)).$ To this end, we will follow an idea in~\cite[Corollary 5.3]{Dan00} and use the returned adjoint form $\cev{\fra^*}$. Indeed, by assumption $\cev{\fra^*}$ is linearly contractive. Thus one can just repeat the argument in Step 1 and obtain 
	\begin{equation}\label{eq: l2linftyreturend}
	\|\cev{U^*}(t,s)\|_{\L(L^{2}(X),L^{\infty}(X))}\leq { \tilde C}(t-s)^{-\frac{\kappa_1 N}{4}}e^{\alpha_{1}^*(t-s)}e^{\tilde \mu\alpha_{2}^*(t-s)} \quad
	\end{equation} 
	for each $(t,s)\in\Delta$ and some constants $\tilde C, \tilde\mu$ that depend only on $N, \kappa_1$. This yields, in turn, an estimate of $U(t,s)$ from $L^2(X)$ to $L^1(X)$, thanks to~\eqref{eq: l2linftyreturend}.
	 Finally, using again the evolution law satisfied by $U$ we conclude that $\U$ is ultracontractive and~\eqref{psedoUltracontractivity} holds. 
	\end{proof}

\begin{proof}[Proof of Theorem~\ref{Thm2: Ultracontractivity}]
Upon rescaling the evolution family by $e^{-\omega(t-s)}$ we can without loss of generality assume $\U$ to be contractive.
Let $q>2, f\in L^q(X)\cap L^2(X)$ and set \[\widehat U(t,s):=e^{-(t-s)[\alpha_1+q\alpha_2]}U(t,s) \quad \text{ for each } (t,s)\in\Delta.\] Because of \eqref{eq them:quasilpcontractivity C0S-2} we have that $t\mapsto\|\widehat U(t,s)f\|_{L^q(X)}$ is decreasing on $[s,T]$ for each $s\in [0, T).$
Let now $(t,s)\in \Delta$: the contractivity of $\U$ together with the Gagliardo--Nirenberg inequality and \eqref{poduct rule classique} imply that for all $q>2$ and all 
	\[
	\begin{split}
	(t-s)\|\widehat U(t,s)f\|_{L^q(X)}^{\frac{4q}{N(q-2)}}&\leq \int_s^t \|\ \widehat U(r,s)f\|_{L^q(X)}^{\frac{4q}{N(q-2)}} dr
	\\&\leq C_G^{\frac{4q}{N(q-2)}} \int_s^t \| \widehat U(r,s)f\|_{L^2(X)}^{\frac{4q}{N(q-2)}-2}\|\ \widehat U(r,s)f\|_{V}^{2} dr
	\\&\leq C_G^{\frac{4q}{N(q-2)}}\alpha^{-1}\|f\|_{L^2(X)}^{\frac{4q}{N(q-2)}-2}\int_s^t \Big[\Re\fra\big(r; \widehat U(r,s)f\big)
+[\alpha_1+q\alpha_2]\Big\|\widehat U(r,s)f\Big\|_{L^2(X)} \Big] dr
	\\&= -C_G^{\frac{4q}{N(q-2)}}\alpha^{-1}\|f\|_{L^2(X)}^{\frac{4q}{N(q-2)}-2}\int_s^t \Re<\frac{\partial}{\partial r}U(r,s)f,U(r,s)f> dr
	\\&=- C_G^{\frac{4q}{N(q-2)}}\alpha^{-1}2^{-1}\|f\|_{L^2(X)}^{\frac{4q}{N(q-2)}-2}\int_s^t \frac{\partial}{\partial r}\|U(r,s)f\|_{L^2(X)}^2 dr
	\\&=- C_G^{\frac{4q}{N(q-2)}}\alpha^{-1}2^{-1}\|f\|_{L^2(X)}^{\frac{4q}{N(q-2)}-2}\Big(\|U(t,s)f\|_{L^2(X)}^2-\|f\|_{L^2(X)}^2\Big)
	\\&\leq C_G^{\frac{4q}{N(q-2)}}\alpha^{-1}\|f\|_{L^2(X)}^{\frac{4q}{N(q-2)}}.
	\end{split}
	\]
	Here we have used that $t\mapsto \widehat U(t,s)f$ solves \eqref{evolution equation u(s)=x} with $A(t)$ replaced by $\A(t)+[\alpha_1+q\alpha_2].$
	It follows that for all $(t,s)\in\Delta$ and all $q>2$ 
	\begin{equation}\label{proof: G-N eq 1} 
	\|U(t,s)\|_{\L( L^2(X), L^q(X))}\leq C_G \alpha^{-N\frac{q-2}{4q}}(t-s)^{-N\frac{q-2}{4q}}e^{(t-s)[\alpha_1+q\alpha_2]}
	\end{equation}
and likewise
	\begin{equation}\label{proof: G-N eq 1 returned} 
	\|\cev{U^*}(t,s)\|_{\L( L^2(X), L^q(X))}\leq C_G \alpha^{-N\frac{q-2}{4q}}(t-s)^{-N\frac{q-2}{4q}}e^{(t-s)[\alpha^*_1+q\alpha^*_2]}
	\end{equation}
Choosing now $q=\frac{2(N+2)}{N+2-2}=\frac{2(N+2)}{N}$ in~\eqref{proof: G-N eq 1}-\eqref{proof: G-N eq 1 returned} we obtain that \eqref{eq: l2l2*Returned adjoint}-\eqref{eq: l2l2*} are fulfilled with $\kappa=C_G\alpha^{-\frac{N}{2(N+2)}}$ and $\kappa_1=\frac{N}{N+2}.$ Thus we conclude by Lemma \ref{Thm1: Ultracontractivity} that $\U$ is ultracontractive and~\eqref{2Ultracontractivity for closed form} holds.
	\end{proof}

	\begin{remark}\label{rem:kurz}
Theorem~\ref{Thm2: Ultracontractivity} holds in particular for $N\frac{q-2}{2q}=1$: in this case the Gagliardo--Nirenberg inequality becomes 
\[
\|u\|_{L^q(X)}\le C_G \|u\|_V\qquad \hbox{for all }u\in V\ ,
\]
i.e.,~\eqref{Gagliardo--Nirenberg inequality} reduces to the elementary assumption that $V$ is continuously embedded in some $L^q(X)$: a classical Sobolev inequality. More precisely, if there exists $N>2$ such that
	\begin{equation}\label{embedding V in Lq}
	 V\subset L^{q}(X) \text{ for } \frac{1}{q}=\frac{1}{2}-\frac{1}{N},
	 \end{equation} 
then $\U$ is ultracontractive and \eqref{2Ultracontractivity for closed form} holds.
	\end{remark}

\section{Gaussian bounds}\label{sec:gaussian}

The existence of integral kernels of the evolution family, established in the previous section, paves the way to the discussion of kernel estimates.



\begin{definition}\label{Definition: Gaussian bound for EVF}
Let $\U$ be an evolution family on $L^2(\R^d)$ with an integral kernel $\Gamma$.
Then $\U$ is said to satisfy \emph{Gaussian bounds} if there exist $b,c>0$, $n>0$, and $\omega\in \R$ such that
	\begin{equation}
	\label{eq:pseudoGaussian bounds}
	|\Gamma(t,s;x,y)|\displaystyle\leq ce^{\omega (t-s)}(t-s)^{-\frac{n}{2}} \exp\big(-b\frac{|x-y|^2}{t-s}\big)
	\end{equation}
	for all $(t,s)\in \Delta$ and a.e.\ $x,y\in\R^d$.
\end{definition}

We regard $L^2(\Omega)$ as a closed subspace of $L^2(\R^d)$, extending operators on $L^2(\Omega)$ to $L^2(\R^d)$ by $0$. In this way we can naturally define Gaussian bounds for operators on $L^2(\Omega)$. 

%
Gaussian bounds for evolution equations can be characterized by ultracontractivity. Well-known for autonomous closed forms, this characterization is based on the so-called \textit{Davies' trick}, first appeared in~\cite{Dav87}, see also~\cite[Thm.~3.3]{AreEls97} and \cite[Thm.~13.1.4]{Are06} for more general versions. Davies' trick is essentially an algorithm centered around an auxiliary result, whose non-autonomous counterpart is Theorem~\ref{Thm: Gaussian bounds} below.

To begin with we introduce a suitable space
\begin{equation}\label{W-Davies Space} 
W:=\Big\{ \psi\in C^\infty(\R^d)\cap L^\infty(\R^d) \mid \|D_j \psi\|_\infty\leq 1, \|D_iD_j \psi\|_\infty\leq 1, i,j=1,\ldots,n \Big\}
\end{equation}
of smooth functions. By \cite[p. 200--202]{Rob91}, the function $d:\R^d\times \R^d\to \R_+$ defined by
\[
d(x,y):=\sup\{|\psi(x)-\psi(y)|\ |\ \psi\in W\}
\]
is a metric equivalent to the Euclidean one: there exists $\beta>0$ such that
 \begin{equation}\label{eq Definition constant beta}
	\beta|x-y|\leq d(x,y)\leq \beta^{-1}|x-y|\qquad \hbox{for all }x\in \R^d.
	\end{equation} 
Let $\U$ be an evolution family on $L^2(\Omega)$ and, as usual, extend it if needed to $L^2(\R^d)$. For a fixed $\psi\in W$ we define perturbed evolution families $\U_\rho$ on $L^2(\R^d)$ by
\[ U_\rho(t,s):=U_\rho^\psi(t,s):=M_\rho U(t,s)M_\rho^{-1},\quad \rho \in \R,\]
where $M_\rho$ is the isomorphism on $L^2(\R^d)$ defined by
\[
(M_\rho g)(x):=(M_\rho^\psi g)(x):=e^{-\rho\psi(x)}g(x),\quad g\in L^2(\R^d),\ x\in \R^d.
\]
Gaussian bounds for $\U$ can now be derived from uniform ultracontractivity of the perturbed 
evolution families $\U_\rho$ with respect to $\rho$ and $\psi.$ 
The proof of this fact is very similar to that of the autonomous case studied in \cite[Prop.~3.3]{AreEls97} and we omit it: our result contains~\cite[Thm.~6.1]{Dan00} as a special case.

\begin{theorem}\label{Thm: Gaussian bounds} Let $\U$ be an evolution family on $L^2(\Omega).$ Then the following are equivalent:
	\begin{enumerate}[(i)]
		\item There exist $c>0$, {$n>0$,} and $\tilde\omega\in\R$ such that
		\begin{equation}\label{Thm:eq1 Gaussian bounds}
		\|U_\rho(t,s)\|_{\L(L^1(\R^d),L^\infty(\R^d))}\leq c(t-s)^{-\frac{n}{2}}e^{\tilde\omega(1+\rho^2)(t-s)}
		\end{equation}
		for all $\rho\in\R$, $\psi\in W$, and $(t,s)\in \Delta$.
		\item $\U$ satisfies Gaussian bounds.
	\end{enumerate}
	In this case $\U$ satisfies~\eqref{eq:pseudoGaussian bounds} with $b=\frac{\beta^2}{4\tilde\omega}, c,\tilde\omega$ as in \eqref{Thm:eq1 Gaussian bounds} and $\beta$ in~\eqref{eq Definition constant beta}.
\end{theorem}

The form domain $V$ of $\fra\in \Formm([0,T];V,H)$ is said to be \textit{$W$-invariant} if $M_\rho V\subset V$ for all $\rho\in \R$ and
$\psi\in W.$
In this case the family of mappings $\fra^\rho$ given by 
\begin{equation}\label{eq:brf}
a^\rho(t;u,v):=a(t;M_\rho u, M_\rho^{-1} v),\quad \rho\in \R,\ t\in [0,T],\ u,v\in V,
\end{equation}
is well-defined. 
Let now $\fra^\rho\in \Formm([0,T];V,H)$: for each $\rho\in \R$, $\psi\in W$, and $t\in[0,T]$ we can hence consider the operator family
\[
\begin{split}
A^\rho(t)&:=M_\rho^{-1}A(t)M_\rho,\\
D(A^\rho(t))&:=\Big\{ u\in L^2(\Omega) \mid M_\rho u\in D(A(t))\Big\},
\end{split}
\]
and $-A^\rho(t)$ is for all $\rho\in \R$ and all $t\in [0,T]$ the generator of the semigroup $\mathcal T_t^\rho$ given by 
\[T_t^\rho(r):=M^{-1}_\rho e^{-rA(t)}M_\rho,\qquad r\geq 0.\]
\begin{lemma}\label{Lemma: associated perturbed closed form} Assume that $V$ is $W$-invariant and $\fra^\rho\in \Formm([0,T];V,H)$ for each $\rho\in \R$ with constants $M_\rho, \alpha_\rho>0$ and $\omega_\rho\in\R$, i.e.,	\begin{equation}\label{eq:rhoellipt}
	\begin{split}
	|a^\rho(t;u,v)|&\leq M_\rho\|u\|_V\|v\|_V \\
	\Re a^\rho(t;u,u)&+\omega_\rho\|u\|^2_H\geq \alpha_\rho\|u\|_V^2
	\end{split}
	\qquad \hbox{for all } t\in[0,T], u,v\in V.
	\end{equation}		
	Then $(A^\rho(t))_{t\in [0,T]}$ and $\U_\rho$ are the operator family and the evolution family on $H$ associated with $\fra^\rho$, respectively. 
\end{lemma}
The easy proof  is left to the reader. 

After all these preparatory results we are finally in the position to present our main theorems: given $\fra\in \Formm([0,T];V,H)$ we introduce two sets of assumptions, which impose a Sobolev-like embedding on $V$ \textit{and} a contractivity condition on the perturbed semigroups $\mathcal T^\rho_t$, and show that each of them imply Gaussian bounds for the evolution family associated with $\fra$.
 \begin{theorem}\label{Main 1 TheoremGaussianEstimate}
Let $\fra\in \Formm([0,T];V,H)$. Assume that
$V$ is $W$-invariant and that~\eqref{eq:rhoellipt} holds for a uniform choice of $\alpha$ and for $\omega_\rho$ such that 
\begin{equation}\label{AdditionalAssumption}
\omega_\rho\leq \omega(1+\rho^2)
\end{equation}
for some constant $\omega>0$ that is independent of $\rho.$
Assume $V$ satisfies a Nash inequality and the semigroups $( e^{-\omega_\rho r}T^\rho_t(r))_{r\geq 0} $ are completely contractive for a.e.\ $t\in [0,T]$ and all $\rho\in\R$.
Then the evolution family $\U$ associated with $\fra$ satisfies Gaussian bounds.
\end{theorem}

\begin{proof}
By Propositions~\ref{proposition Linftycontractivity} and~\ref{prop:L1contractivity} the evolution family $(e^{-{\omega_\rho}(t-s)}U_\rho(t,s))_{(t,s)\in\Delta}$ is completely contractive. Thus by Theorem \ref{Theorem: Ultracontractivity via Nash ineq}
\begin{equation}
\|e^{-{\omega_\rho}(t-s)}U_\rho(t,s)\|_{\L(L^1(\Omega),L^\infty(\Omega))}\leq \Big(\frac{\mu c}{4\alpha}\Big)^{\frac{\mu}{2}}(t-s)^{-\frac{\mu}{2}}e^{\omega_\rho(t-s)}\quad\hbox{for	all }(t,s)\in\Delta.
\end{equation}

Now, using \eqref{AdditionalAssumption} we obtain that $\U_\rho$ satisfies \eqref{Thm:eq1 Gaussian bounds}: the claim follows from Theorem \ref{Thm: Gaussian bounds}.
\end{proof}


 \begin{theorem}\label{Main 2 TheoremGaussianEstimate}
 	Let $\fra\in \Formm([0,T];V,H)$ with associated evolution family $\U$. Assume that
 	$V$ is $W$-invariant and that~\eqref{eq:rhoellipt} holds for a uniform choice of $\alpha$ and all $t,u,v.$ Assume that $V$ satisfies a Gagliardo--Nirenberg inequality and both $\U_\rho$ and $\cev{{\U_\rho}^*}$ are linearly quasi-contractive for all $\rho\in\R.$ Then $\U_\rho$ is ultracontractive for all $\rho\in\R$ with 
 	\begin{equation}
 	\|U_\rho(t,s)\|_{\L(L^{1}(\Omega),L^{\infty}(\Omega))}\leq C_G^{\frac{N}{2}}\alpha^{-\frac{N^2}{4(N+2)}}e^{(t-s)\tilde\omega_\rho}(t-s)^{-\frac{N^2}{2(N+2)}}\quad \hbox{for all }(t,s)\in\Delta,
 	\end{equation}
 	where
 	\[
 	\tilde\omega_\rho:=\left(\omega_{\rho}+\alpha_{\rho,1}+\alpha_{\rho,1}^*+\frac{2(N+2)}{N}[\alpha_{\rho,2}+\alpha_{\rho,2}^*]\right)
 	\]
and $\alpha_{\rho,i},\alpha^*_{\rho,i}$ are the constants that appear in the linear quasi-contractivity estimate. Thus, if additionally $\omega_\rho, \alpha_{\rho,i}, \alpha_{\rho,i}^*, i=1,2,$ can be chosen in such a way that 
\begin{equation}\label{eq:omegarhoomega0}
\tilde\omega_\rho\le \omega_0(1+\rho^2)
\end{equation}
 	for some constant $\omega_0>0$ independent of $\rho$, then $\U$ satisfies Gaussian bounds. 
 \end{theorem}
\begin{proof}
	The assertion can be proved similarly to Theorem~\ref{Main 1 TheoremGaussianEstimate}, based in this case on Theorem \ref{Thm: Gaussian bounds} and Theorem \ref{Thm2: Ultracontractivity}.
\end{proof}




\section{Applications}\label{sec:appl}
\subsection{Diffusion equations on dynamic graphs}\label{sec:diff-evolving}


Consider a (finite or infinite) simple graph $\mG$ with vertex set $\mV$ and edge set $\mE$, with $V$ vertices and $E$ edges (i.e., $V=|\mV|$ and $E=|\mE|$). Fix an orientation of $\mG$ and introduce the $V\times E$ (signed) \textit{incidence matrix} $\mathcal I=(\iota_{\mv\me})$ of $\mG$ by
\[
{\iota}_{\mv \me}:=\left\{
\begin{array}{ll}
-1 & \hbox{if } \mv \hbox{ is initial endpoint of } \me, \\
+1 & \hbox{if } \mv \hbox{ is terminal endpoint of } \me, \\
0 & \hbox{otherwise}.
\end{array}\right.
\]
Let $m\in \ell^\infty(\mE)$ be a family of edge weights and consider the (weighted) Laplacian $\mathcal L:=\mathcal I \mathcal M \mathcal I^T$ on $\ell^2(\mV)$, where $0\le \mathcal M:=\diag(m(\me))_{\me\in \mE}$. ($\mathcal L$ can be shown to be independent of the orientation.)

We assume that $\mG$ is \textit{uniformly locally finite}, i.e., there is $M<\infty$ such that $\sum_{\me\in \mE} |\iota_{\mv\me}|\le M$ for all $\mv\in\mV$: in this case $\mathcal I$ is a bounded linear operator from $\ell^2(\mE)$ to $\ell^2(\mV)$~\cite[Lemma.~4.3]{Mug14}, hence
$\mathcal L$ is a positive semi-definite, bounded self-adjoint operator on $\ell^2(\mV)$: we can thus take $V=H=V'=\ell^2(\mV)$. It is well-known that the semigroup generated by $-\mathcal L$ is sub-Markovian, see e.g.~\cite[\S~6.4.1]{Mug14}; if the graph is finite, then it is Markovian and stochastic, too.

Let us now regard $\mG$ as a reference graph (one may e.g.\ think of a complete graph, or else of a lattice graph $\mathbb Z^d$) and consider a family $(\mG(t))_{t\in [0,T]}$ of modifications of $\mG$ -- in other word, a graph-valued dynamical system, or \emph{dynamic graph}~\cite{Sil08}. 
We describe the dependence of $\mG(t)$ on $t$ by introducing a measurable function $[0,T]\ni t\mapsto m(t)\in \ell^\infty(\mE)$:
this allows e.g.\ for 
sudden switching of edges is allowed (as in the case of adjacency driven by a Poisson process). In particular, we consider the non-autonomous form $\fra$ defined by
\[
a(t;u,v):=\left(\diag(m(t,\me))\mathcal I^T u \mid \mathcal I^T v\right)_{\ell^2(\mE)},\qquad t\in [0,T],\ u,v\in \ell^2(\mV).
\]
It is easy to see that $\fra\in \Formm([0,T];\ell^2(\mV),\ell^2(\mV))$ and the associated operators are the Laplacians $(\mathcal L_{\mG(t)})_{t\in [0,T]}$.
 (We are not assuming boundedness from below on $m$: this is made unnecessary by the boundedness of the operator $\mathcal L_{\mG(t)}$ for all $t$; in fact, even negative weights and hence signed graphs are allowed.) We deduce by Proposition~\ref{cor:easytho} that the non-autonomous Cauchy problem
\[
 \begin{split}
\dot{u}(t,\mv)+\mathcal L_{G(t)}u(t,\mv)&=0,\quad t\ge s, \ \mv\in\mV,\\
u(s,\mv)&=x_\mv,\quad \mv\in \mV
\end{split}
\]
is governed by an evolution family on $\ell^2(\mV)$; in fact, for all $x\in \ell^2(\mV)$ the above equation enjoys backward well-posedness, too, and the unique solution $u$ is of class $H^1(\R;\ell^2(\mV))$: the corresponding evolution family $(U(t,s))_{(t,s)\in \R^2}$ can be defined via product integrals. As observed in~\cite[Example~7.3]{AreDie18} $U(t,s)$ is sub-Markovian for all $t,s$; in particular, $\U$ extrapolates to a consistent family of contractive evolution families on $\ell^p(\mV)$ for all $p\in [1,\infty]$. The evolution family is also positivity improving, and additionally stochastic if $\mG$ is finite. Furthermore, \cite[Thm.~2.6]{Laa18} yields that the evolution family is immediately norm-continuous if $[0,T]\ni t\mapsto m(t)\in \ell^\infty(\mE)$ is H\"older continuous with exponent $\alpha>1/2$; by~\cite[Thm.~7.4.1]{Fat83} it is even holomorphic if additionally $[0,T]\ni t\mapsto m(t)\in \ell^\infty(\mE)$ extends to a holomorphic function on an open convex neighborhood in $\mathbb C$ of $[0,T]$.
 
To conclude, let us study Laplacians on subgraphs $\mG_t$ induced by subsets $\mV_t$ of $\mV$ as in~\cite[Chapt.~8]{Chu97} in the unweighted case ($m(t,\me)\in \{0, 1\}$). Even in the autonomous case, Laplacian on (non-trivial) subgraphs of $\mG$ generate semigroup that neither are dominated by, nor dominate $(e^{-r\mathcal L_\mG})_{r\ge0}$: this can e.g.\ be seen by applying~\cite[Cor.~2.22]{Ouh05}. Things change, however, if Dirichlet boundary conditions are imposed, e.g., if $\mathcal L$ on $\mG$ is restricted to 
\[
D_t:=\{f\in \ell^2(\mV): f_{|\mV\setminus \mV_t}\equiv 0\},\qquad t\in [0,T].
\]
Because $D_t$ is for all $t$ a generalized ideal of $V=\ell^2(\mV)$, the associated Laplacian $\mathcal L_{|D_t}$ generates for all $t$ a semigroup $(e^{-r\mathcal L_{|D_t}})_{r\ge 0}$ which is -- again by~\cite[Cor.~2.22]{Ouh05} -- dominated by $(e^{-r\mathcal L})_{r\ge 0}$. Therefore, by Proposition~\ref{cor:easytho} and Proposition~\ref{prop:domin} the evolution family $(U(t,s))_{(t,s)\in \Delta}$ satisfies
\begin{equation}\label{eq:domingra}
|U(t,s)f|\le e^{-(t-s)\mathcal L}|f|\qquad \hbox{for all $(t,s)\in\Delta$ and $f\in H$}.
\end{equation}

We have seen in the introduction that if $A(t)\equiv A$, then the evolution family that governs the non-autonomous problem is given by $U(t,s)=e^{(t-s)A}$, 
hence $\U$ satisfies Gaussian bounds if and only if so does $(e^{rA})_{r\ge 0}$. We can now show a less trivial instance of Gaussian-type bounds.

Gaussian-type kernel estimates on $(e^{-t\mathcal L_{\mathcal G}})_{t\ge 0}$ have been proved in~\cite{Del99} for certain classes of $\mG$. Thus, if $(\mG_t)_{t\in [0,T]}$ is a family of subgraphs of a reference graph $\mG$ with measurable $t\mapsto m(t,\me)$ for all $\me\in\mE$, and if $\U$ is the evolution family associated with the corresponding Laplacians $-\mathcal L_{|D_t}$, then~\eqref{eq:domingra} yields a Gaussian-type kernel estimate. If we e.g.\ take $\mG$ to be $\Z$, then
\[
0\le \Gamma(t,s;n_1,n_2)\le G(t-s;n_1,n_2),\qquad (t,s)\in\Delta,\ n_1,n_2\in \Z,
\]
where 
\[
G(r;n_1,n_2):=\frac{1}{2\pi} \int_{-\pi}^\pi \cos((n_1-n_2)q)e^{-2r(1-\cos q)}\d q
\] 
is the heat kernel on $\Z$ explicitly computed e.g.\ in~\cite[Exa.~12.3.3]{Dav07}.

\subsection{Time-dependent pageranks}

Let us study a model similar to that of Example~\ref{sec:diff-evolving}: it is based on an idea proposed in~\cite{Chu07}, cf.~\cite{Gle15} for later developments, where the connectivity of $\mG$ describes the  links within a server network -- possibly the whole World Wide Web.

We thus consider an orientation of a \textit{finite} complete graph (i.e., a graph such that either $(\mv,\mw)\in \mE$ or $(\mw,\mv)\in \mE$ for any $\mv,\mw\in\mV$ with $\mv\ne \mw$). 
As in~\ref{sec:diff-evolving}, we assign a weight $m$ to each edge: if e.g.\ $m(t,\me)\in \{0,1\}$ for all $\me\in\mE$ and all $t\in[0,T]$, then we are effectively shutting off/switching on certain links in the considered network. We then consider the matrix 
\begin{equation}\label{eq:pagerdef}
A:=\mathcal I \mathcal M(\mathcal I^-)^T (\mathcal D^{\rm out})^{-1}
\end{equation}
where $\mathcal I$ is again the incidence matrix of $\mG$ (see Example~\ref{sec:diff-evolving}), $\mathcal I^-:=(\iota^-_{\mv\me})$ is its negative part, $\mathcal M=\diag(m(\me))_{\me\in \mE}$, and $\mathcal D^{\rm out}:=\diag (\deg^{\rm out}(\mv))_{\mv\in \mV}$, where $\deg^{\rm out}(\mv):=\sum_{\me\in \mE}|\iota^-_{\mv\me}m(\me)|$.

Then, $A$ defines a so-called \textit{heat kernel pagerank} $e^{-rA}x
$ of $\mG$ with parameters $r$ and $x$: here $r$ is a positive time and $x$ a probability distribution on $\mV$, i.e., $x\in \R^\mV$, $x_\mv\ge 0$ for all $\mv\in \mV$ and $\|x\|_1=1$. The rationale behind this definition is that $A$ is a column stochastic matrix, hence $(e^{-rA})_{r\ge 0}$ is a stochastic semigroup and $e^{-rA}f$ is thus again a probability distribution for all $r\ge 0$, which can be used to measure the relevance of a certain node within a network in a way similar to Google's classical PageRank, cf.~\cite[\S~2.1.7.3]{Mug14}.

We can now consider a measurable function $t\mapsto (m(t,\me))_{\me\in \mE}$ and accordingly a time-dependent matrix family $(A(t))_{t\in [0,T]}$ as in~\eqref{eq:pagerdef} these matrices will in general not be symmetric, but in view of finiteness of $\mV$ they are certainly associated with a form $\fra\in \Formm([0,T];\ell^2(\mV),\ell^2(\mV))$. Accordingly, in view of Proposition~\ref{them: stochastic EVF} the associated evolution family $(\mathcal U(t,s))_{(t,s)\in \Delta}$ consists of stochastic operators and hence $\mathcal U(t,s)f$ is a probability distribution on $\mV$ for all $(t,s)\in \Delta$ and all probability distributions $f\in \R^\mV$.

There is a correspondence between linear transport differential equations on networks and flows on their underlying graphs~\cite{Dor05}: accordingly, our results also extend to the space-continuous case. The well-posedness result in~\cite[\S~6]{Bay12} -- which relies on the assumption that the dependence of the graph on time is absolutely continuous -- can thus be strengthened: we omit the details.



\subsection{Black--Scholes equation with time-dependent volatility}\label{sec:blacks}

The Cauchy problem consisting of the backward parabolic equation
\begin{equation}\label{eq:bs-m}
u_t (t,x) + \frac12 x^2 \sigma^2 u_{xx} (t,x) + rxu_x (t,x) - ru(t,x) = 0,\qquad t\in [0,\tau],\ x\in ]0,\infty[,
\end{equation}
along with the final value assignment
\[
u(\tau,x)=h(x)\qquad x\in ]0,\infty[
\]
was derived in~\cite{BlaSch73} and is currently considered among the main mathematical tool in the pricing theory of European options: the positive constants $\sigma,r$ describe volatility and interest rate of the system, respectively, whereas $\tau$ is the maturity time of an option.

An effective variational approach to the relevant operator appearing in the Black--Scholes equation has been discussed in~\cite{Ein08}: it is based on studying the sesquilinear form
\begin{equation}\label{eq:eineform}
\begin{split}
a (u,v)&:=\frac{\sigma^2}{2} \int_0^\infty x^2 u'(x) \overline{v'(x)} \ \d x\\
&\qquad +(\sigma^2-r) \int_0^\infty x u'(x)\overline{v(x)}\ \d x+r \int_0^\infty u(x)\overline{v(x)}\ \d x,\qquad u,v\in V,
\end{split}
\end{equation}
defined on the form domain
\[
V:=\{u\in W^{1,1}_{loc}]0,\infty[\cap L^2]0,\infty[:\ \id \cdot u'\in L^2]0,\infty[ \},
\]
which is a Hilbert space with respect to the inner product
\[
(u|v)_V:=\int_0^\infty x^2 u'(x)\overline{v'(x)}\ \d x+\int_0^\infty u(x)\overline{v(x)}\ \d x.
\]
Then it was proved in~\cite[\S~7.2]{Ein08} that $V$ is continuously dense in $L^2]0,\infty[$ and that furthermore
\begin{equation}\label{eq:einemann}
\begin{split}
|a(u,v)|&\le \left(\frac{\sigma^2}{2}+|\sigma^2-r|+|r|\right)\|u\|_V \|v\|_V\\
\Re a(u,u)+\left(\sigma^2-\frac{3r}{2}\right) \|u\|^2_{L^2}&=\frac{\sigma^2}{2}\|u\|^2_V
\end{split}
\qquad\hbox{for all }u\in V,
\end{equation}
i.e., $a$ is bounded and elliptic; and $\Re a(u,u)\ge 0$ if $3r\ge \sigma^2$.

The original Black--Scholes-theory assumes $\sigma$ to be time-independent, but is rather unrealistic and has been questioned ever since: we mention the celebrated Heston model~\cite{Hes93}, which leads to a non-autonomous PDE similar to~\eqref{eq:bs-m}, based on the assumption that the volatility evolves following a certain Brownian-like motion. This justifies the study of
\[
u_t (t,x) + \frac12 x^2 \sigma^2(t) u_{xx} (t,x) + rxu_x (t,x) - ru(t,x) = 0,\qquad t\in [0,\tau],\ x\in ]0,\infty[,
\]
with measurable dependence $t\mapsto \sigma(t)$. The computations in~\eqref{eq:einemann} show that if $0<\sigma_0\le \sigma(t)\le \sigma_1$ for a.e.\ $t\in [0,\tau]$, then $\fra\in \Formm([0,T];V;H)$, where $\fra$ is defined by
\[
\begin{split}
a (t;u,v)&:=\frac{\sigma^2(t)}{2} \int_0^\infty x^2 u'(x) v'(x) \ \d x\\
&\qquad +(\sigma^2(t)-r) \int_0^\infty x u'(x)v(x)\ \d x+r \int_0^\infty u(x)v(x)\ \d x,\qquad u,v\in V.
\end{split}
\]
Furthermore, the semigroup associated with $a$ is quasi-contractive and sub-Markovian: we deduce from Proposition~\ref{cor:easytho} that such non-autonomous Black--Scholes equation is governed by a sub-Markovian evolution family $\U$ that extrapolates to all $L^p]0,\infty[$ spaces, $p\in [2,\infty]$. In view of Proposition~\ref{cor:easytho}  we can also apply~\cite[Thm.~7.2.5]{Ein08} and deduce that $\U$ leaves invariant the order interval $]-\infty,\id]_H$. By~\cite[Rem.~7.2.4]{Ein08} $\min\{\log,0\}\in V$, hence $V\not\hookrightarrow L^\infty]0,\infty[$; however, it is unclear whether $V$ satisfies a Nash or Gagliardo--Nirenberg inequality, which would imply ultracontractivity of the evolution family.

 (A manifold of financial models exist that display a similar mathematical structure, albeit their meaning is different: the popular Cox--Ingersoll--Ross along with several other so-called \textit{short-rate models} surveyed in~\cite{ChaKarLon92} involve time-dependent $\sigma$ and/or $r$ and can be discussed with only minor variations to our treatment above.)

\subsection{Second-order elliptic operators on networks}\label{sec:networks}

With the purpose of introducing a differential operator on a network-like structure, we consider like in Example~\ref{sec:diff-evolving} a possibly infinite, but uniformly locally finite graph $\mG=(\mV,\mE)$ and identify each edge $\me\in\mE$ with an interval $[0,1]$. In other words, we are considering a collection of copies of $[0,1]$ and gluing them in a graph-like fashion: we thus obtain what are often called \textit{metric graphs} or \textit{networks} in the literature~\cite{BerKuc13,Mug14}.  (For the sake of simplicity we are going to assume such a network to be connected.)
The history of non-autonomous diffusion equations on networks goes back at least to pioneering investigations by von Below, Lumer, and Schnaubelt: well-posedness results could be proved in~\cite{Bel88,LumSch99}, further results on long-time asymptotics have been deduced in~\cite{AreDieKra14}.

We are going to apply in this context the theory developed in the previous sections: on each interval $\me\simeq [0,1]$ we consider the operator family
\[
A_\me(t):u_\me\mapsto -\frac{\d}{\d x}\left(c_e(t) \frac{\d}{\d x}u_e\right)-p_e(t) u_e,\qquad t\in [0,T].
\]
We assume the coefficients
\[
[0,T]\ni t\mapsto c_\me(t)\in L^\infty(0,1;\ell^\infty(\mE))\quad \hbox{and}\quad [0,T]\ni t\mapsto p_\me(t)\in L^1(0,1;\ell^\infty(\mE))
\]
to be measurable: this defines in a natural way an operator $A$ with domain $D(A):=\widetilde{H^2}(\mathcal G):=H^2(0,1;\ell^2(\mE))$ on the Hilbert space
$H:=L^2(\mathcal G):=L^2(0,1;\ell^2(\mE))$.
We will additionally assume that the operator family is uniformly elliptic, i.e.,
\begin{equation}\label{Eq: uniformly ellipticityGraph}
c_\me(t,x)\ge \gamma \qquad \hbox{for all }\me\in \mE\hbox{ and a.e. }x\in (0,1),\ t\in [0,T],
\end{equation}
for some $\gamma>0$.
In order to reflect the topology of the graph, transmission conditions in the vertices are required: the most common conditions are usually referred to as \textit{continuity/Kirchhoff} and amount to asking that 
\begin{itemize}
\item $u$ is continuous, i.e., the boundary values of $u_\me$ and $u_\mf$ agree whenever evaluated at endpoints of the intervals $\me,\mf$ that are glued together in the network $\mathcal G $ (continuity); 
\item $u$ satisfies a Kirchhoff-type rule, i.e., at any vertex $\mv$ the sum over all neighboring edges of the normal derivatives evaluated at $\mv$ vanishes.
\end{itemize}

However, more boundary conditions are conceivable: a parametrization of an infinite class of boundary conditions that fits well the setting of sesquilinear forms has been discussed in~\cite[\S~6.5.1]{Mug14}, based on the finite case treated in~\cite[Thm.~5]{Kuc04}. Fix a closed subspace $Y$ of the Hilbert space $\ell^2(\mE)\times \ell^2(\mE)$, let $(\Sigma(t))_{t\in [0,T]}$ be a family of bounded linear operators on $Y$, and let
 \[
 \underline{u}:=\begin{pmatrix}(u_{e}(0))_{\me\in\mE}\\ (u_{e}(1))_{\me\in\mE}\end{pmatrix},\quad 
\underline{\underline{u}}:=\begin{pmatrix}(-c_\me (0)u'_{\me}(0))_{\me\in\mE}\\ (c_\me(1) u'_{\me}(1))_{\me\in\mE})\end{pmatrix}.
\]
We can then consider the non-autonomous form $\fra$ defined by
\[
a(t;u,v)=
\int_0^1 \left[ \left(c_\me(t,x)u'_\me(x)|v'_\me(x)\right)_{\ell^2(\mE)}+\left(p_\me(t,x) u_\me(x)| v_\me(x)\right)_{\ell^2(\mE)}\right]\textrm{d}x +(\Sigma(t) \underline{u}|\underline{v})_Y
	\]
	with time-independent form domain
	\[
	H^1_Y(\mathcal G):=\left\{u\in \bigoplus_{\me\in\mE}H^1(0,1;\ell^2(\mE)):\underline{u}\in Y\right\}.
\]
Then the conditions in the vertices satisfied by functions in the domain of each operator $A(t)$ associated with $\fra$ can be written in a compact form as
\begin{equation}\label{eq:Ycond}
\underline{u}\in Y\quad \hbox{and}\quad \underline{\underline{u}}+\Sigma(t) \underline{u}\in Y^\perp.
\end{equation}
We finally assume that for some $P,S>0$
\[
\|p(t)\|_{L^1}\le P\quad\hbox{and}\quad \|\Sigma(t)\|_{\mathcal L(Y)}\le S\qquad\hbox{for all }t\in [0,T].
\]
Using an obvious extension of \cite[Lemma 6.22]{Mug14} to non-autonomous forms we see that $\fra\in\Formm([0,T];V,H)$. Our abstract results in the previous sections hence yield the following.
\begin{proposition}\label{prop:netw-well}
Under the above assumptions on the coefficients $c_\me,p_\me$, the space $Y$, and the operators $\Sigma$, the form $\fra$ is associated with a strongly continuous evolution family $\U$ on $H$.
If $p_\me(t)\ge 0$ and $\Sigma(t)$ is accretive for a.e.\ $t\in [0,T]$, then
$\mathcal U$ is contractive.

If all these coefficients are defined on the whole interval $[0,\infty[$, then $\U$ extends to an evolution family on $\{(t,s):0\le s\le t<\infty\}$.
\end{proposition}

We denote by $P_Y$ the orthogonal projector of $\ell^2(\mE)\times \ell^2(\mE)$ onto $Y$; the latter inherits the lattice structure of $\ell^2(\mE)\times \ell^2(\mE)$. Owing to Proposition~\ref{proposition Linftycontractivity}  we can formulate the following generalization of~\cite[Thm.~6.85]{Mug14} (see also~\cite[Prop.~5.1]{CarMug09}).

%

%
%

\begin{corollary}\label{cor:netw-subm}
(1) If $e^{-r\Sigma(t)}$ (for all $r\ge 0$ and a.e.\ $t\in [0,T]$) and $P_Y$ are positive, then $\U$ is positive. If additionally $p_\me\equiv 0$, $e^{-r\Sigma(t)}$ is (sub-)stochastic and ${\bf 1}\in Y$, then $\U$ is (sub-)stochastic.

(2) Let $p_\me(t,x)\ge 0$ for all $\me\in \mE$ and a.e.\ $t\in (0,T)$ and $x\in (0,1)$. If $e^{-r\Sigma(t)}$ (for all $r\ge 0$ and a.e.\ $t\in [0,T]$) and $P_Y$ are $\ell^\infty$-contractive, then $\U$ is $L^\infty(0,1;\ell^2(\mE))$-contractive.

(3) Under the assumptions of (2), let additionally $e^{-r\Sigma(t)^*}$ be $\ell^\infty$-contractive for all $r\ge 0$ and a.e.\ $t\in [0,T]$. Then $\U$ is completely contractive; accordingly, it extrapolates  to a strongly continuous, contractive evolution family on  all spaces $L^p(0,1;\ell^2(\mE))$, $p\in [1,\infty]$.
\end{corollary}

\begin{ex}\label{ex:kuchm}
The continuity/Kirchhoff vertex conditions are special cases of the general conditions in~\eqref{eq:Ycond}. Indeed, denote by $c_{\mV}$ the vector in $\ell^2(\mE)\times \ell^2(\mE)$ that consists of vertex-wise constants, i.e., entries of $c_{\mV}$ agree whenever they correspond to endpoints of edges the same vertex $\mv\in \mV$ is incident with. Let by $Y$ the subspace of $\ell^2(\mE)\times \ell^2(\mE)$ 
spanned by $c_{\mV}$ and take $\Sigma=0$: then~\eqref{eq:Ycond} agrees with continuity--Kirchhoff conditions: we denote by
\[
H^1(\mathcal G)
\] 
the Sobolev space $H^1_Y(\mathcal G)$ with respect to this distinguished space $Y$. Under stronger assumptions on $p_\me,c_\me$, a well-posedness result comparable to Proposition~\ref{prop:netw-well} has been obtained in~\cite[Thm.~3.3]{AreDieKra14}.
Because $H^1(\mathcal G)\hookrightarrow C(\mathcal G)$, we furthermore deduce that $U(t,s)$ maps $C(\mathcal G)\cap L^2(\mathcal G)$ into $C(\mathcal G)$ for all $s\in [0,T]$ and a.e.\ $t\in (s,T]$.

Due to the standing assumption that $\mG$ is uniformly locally finite, $P_Y$ is a block operator matrix whose blocks are of the form
$\frac{1}{n}J_n$ ($J_n$ denoting the $n\times n$ all-1-matrix, $n$ the degree of the corresponding vertex). Because $P_Y$ leaves invariant the order interval $[-1,1]_{\ell^2\times \ell^2}$ and $\Sigma\equiv 0$, we deduce that $(e^{-rA(t)})_{r\ge 0}$ is positive and -- if $p_\me \ge 0$ -- sub-Markovian for a.e.\ $t\in [0,T]$; hence by Proposition~\ref{cor:easytho} so is the evolution family $\U$.	

 As concerns the long-time behavior if $T=\infty$, we can hence discuss two cases:

\begin{itemize}
\item If $\liminf\limits_{t\to\infty}\frac{1}{t-t_0}\int_{t_0}^{t} \essinf p_\me(r,\cdot)\ dr>0$ for all $\me\in\mE$, then $\U$ is by Proposition~\ref{prop:ues} uniformly exponentially stable.
\item Let $p_\me\equiv 0$ for all $\me\in\mE$. The network is always assumed to be connected; if it is additionally finite (i.e., $|\mE|<\infty$), then  by Proposition~\ref{cor:stability} and \cite[Prop.~6.70]{Mug14} $0$ is a simple eigenvalue of each $A(t)$; $A(t)$ is self-adjoint and the its null space consists of all constant functions. Furthermore, not only has each $A(t)$ a spectral gap, but there is a uniform lower bound on them: By Nicaise' inequality~\cite[Théo.~3.1]{Nic87}
$\tilde{s}(t)\ge \frac{\gamma^2\pi^2}{|\mE|^2}>0$ for a.e.\ $t\in [0,\infty[$, and we conclude by Proposition~\ref{cor:stability} that
\[
\left\|U(t,s)f-\frac{1}{\sqrt{|\mE|}}\int_{\mathcal G}f \d x\cdot {\bf 1}\right\|\le e^{-\frac{\gamma^2\pi^2}{|\mE|^2}(t-s)}\|f\| \qquad\hbox{for all }f\in H\hbox{ and }0\le s\le t<\infty.
\]
\end{itemize}
Similar convergence results have been obtained in the strong topology for general inhomogeneous diffusion equations in~\cite[\S~5.1]{AreDieKra14}.
\end{ex}

\begin{proposition}\label{prop:netwultra}
Let $p_e(t,x)\ge 0$ for all $\me\in\mE$ and a.e.\ $t\in [0,T]$, $x\in (0,1)$. Let furthermore $P_Y=(\pi^Y_{ij})$ and $\Sigma(t)=(\sigma(t)_{ij})$ satisfy 
\begin{itemize}
\item $\sum_j |\pi_{ij}|\le 1$ for all $i$;
\item $\Re \sigma(t)_{ii}\ge \sum_{j\ne i}|\sigma(t)_{ij}|$ for all $i$ and a.e.\ $t\in [0,T]$;
\item $\Re \sigma(t)_{ii}\ge \sum_{j\ne i}|\sigma(t)_{ji}|$ for all $i$ and a.e.\ $t\in [0,T]$.
\end{itemize}
Then $\U$ is completely contractive. If additionally $Y$ is a generalized ideal of $\langle c_\mV\rangle$, then $\U$ is ultracontractive.
\end{proposition}

\begin{proof}
It follows from~\cite[Lemma~6.1]{Mug07} and Corollary~\ref{cor:netw-subm} that $\U$ is completely contractive. 

Under the assumption that $Y$ is a generalized ideal of $\langle c_\mV\rangle$ it has been shown in~\cite[Chapt.~3]{Pro14} that $H^1(\mathcal G)$ satisfies a Nash inequality whenever $\mathcal G$ is a connected, locally finite metric graph with edge lengths uniformly bounded away from 0: accordingly, the non-autonomous form with domain $H^1(\mathcal G)$ is associated with an ultracontractive $\U$, owing to Theorem~\ref{Theorem: Ultracontractivity via Nash ineq}.

By~\cite[Thm.~6.2]{CarMug09} the semigroup associated with $a(t)\equiv a$ with domain $H^1_Y(\mathcal G)$ is dominated by the  semigroup associated with the same form with domain $H^1(\mathcal G)$, provided $Y$ is a generalized ideal of $\langle c_\mV\rangle$ (in fact by \cite[Thm.~C.II-5.5]{Nag86} the latter is the modulus semigroup of the former one). By Proposition~\ref{cor:easytho}, the same holds for the associated evolution families, hence the former heat kernel inherits ultracontractivity from the latter one.
\end{proof}

The following result seems to be new even in the autonomous case: in~\cite{Mug07} Gaussian bounds for heat kernels on \textit{finite} networks have been proved only in the special case of $Y=\langle c_\mV\rangle$, see also~\cite[Chapt.~7]{Mug14} for an abstract approach based on the theory of Dirichlet forms.

\begin{corollary}\label{cor:general}
Under the assumptions of Proposition~\ref{prop:netwultra}, let $Y$ be $\langle c_\mV\rangle$-invariant (i.e., the entrywise product $\psi c_\mV$ lies in $Y$ for all $\psi\in Y$), where the vector $c_\mV$ is defined as in Example~\ref{ex:kuchm}. If furthermore $\Sigma(t)$ is diagonal for a.e.\ $t\in [0,T]$, then $\U$ satisfies Gaussian bounds.
\end{corollary}

We stress that our assumption on $P_Y$ and $\Sigma$ are only enforcing complete contractivity, but the evolution family need not be positive. An example is given by the non-autonomous parabolic equation on a loop with boundary conditions defined by $\Sigma\equiv 0$ and $Y=\langle{1\choose -1}\rangle$, which is $\langle c_\mV\rangle$-invariant (here $c_\mV={1\choose 1}$): this equation is governed by a completely contractive and (in view of the Nash inequality for $H^1(\mathcal G)$) ultracontractive evolution family $\U$, which therefore enjoys Gaussian bounds. However, $\U$ is not positive, since neither is $P_Y$.

\begin{proof}[Proof of Corollary~\ref{cor:general}]
We apply Davies' Trick in a slightly different version. Indeed, we adapt the usual setting to our network environment by introducing the space
\[
W_{\mathcal G}:=\left\{ 
\psi \in H^1(\mathcal G)\cap C^\infty(0,1;\ell^2(\mE))\ |\ \|\psi'\|_\infty\le 1,\ \|\psi''\|_\infty\le 1\right\}\ :
\]
then one can check that 
\[
d(x,y):=\sup\{|\psi(x)-\psi(y)| \ | \ \psi\in W_{\mathcal G}\},\qquad x,y\in\mathcal G,
\]
defines a metric on $\mathcal G$ that is equivalent to the canonical one~\cite[\S~3.2]{Mug14}. (Recall that $H^1(\mathcal G)$ denotes the space $H^1_{\tilde Y}(\mathcal G)$, where ${\tilde Y}=\langle c_\mV\rangle$ is the space spanned by $c_\mV$: each function in $H^1(\mathcal G)$ is by definition continuous on the metric space $\mathcal G$ and $\langle c_\mV\rangle$ is an algebra with respect to the entry-wise product.) 
By definition, $H^1_Y(\mathcal G)$ is $W_{\mathcal G}$-invariant if and only if $\underline{u}\in Y$ implies $\underline{e^{\rho\psi}u}\in Y$, i.e., if and only if $Y$ is $\langle c_\mV\rangle$-invariant.

If $\Sigma(t)\equiv 0$, then the assertion has been proved in~\cite[Thm.~4.7]{Mug07} by showing that the relevant form (let us denote it by $\fra_0$ to stress the absence of boundary terms) induces perturbed forms $\fra_0^\rho$ that are associated with completely contractive perturbed evolution families (with the form domain being unchanged and still satisfying a Nash inequality). In the general case of $a(t;u,v)=a_0(t;u,v)+(\Sigma(t)\underline{u}|\underline{v})_Y$, we find that $a^\rho(t;u,v)=a_0^\rho(t;u,v)+(\Sigma(t)\underline{u}|\underline{v})_Y$. These forms are associated with completely contractive evolution families, hence the claim follows.
\end{proof}

\subsection{Second-order elliptic operators with complex coefficients on open domains}\label{Example Ellipt Op }
Let $\Omega\subset \R^d$ be an open set. On the complex Hilbert space $L^2(\Omega)$ we consider the non-autonomous form $\fra_V:[0,T]\times V\times V$ defined by
\begin{align}\label{EllipOper-associated form}
a_V(t;u,v):=\sum_{k,j=1}^d\int_{\Omega} a_{kj}(t;x)D_ku\overline{D_j v} \ \d x&+\sum_{k=1}^d\int_{\Omega} \Big[b_k(t;x) D_ku\overline{v} +c_k(t;x)u\overline{D_k v}\Big]\d x \\&\nonumber+ \int_{\Omega} a_0(t;x)u \overline{v}\d x
\end{align}
for all $u,v\in V$ where $V$ is a closed subspace of $H^1(\Omega)$ that contains $H_0^1(\Omega)$. We assume that the coefficients $a_{k,j}, b_k, c_k, a_0$ lie in $L^\infty([0,T]\times\Omega;\mathbb C)$. Moreover, we assume that the principal part is uniformly elliptic, i.e., there exist a constant $\nu>0$ such that
\begin{equation}\label{UniformEllipticityEq}
\Re \sum_{k,j=1}^d a_{kj}(t;x)\xi_k\overline{\xi_j}\geq \nu |\xi|^2\quad \hbox{for a.e. }t\in [0,T], \ x\in \Omega \hbox{ and all }\xi\in \C^n.
\end{equation}
Then $\fra_V$ defined in~\eqref{EllipOper-associated form} belongs to $\Formm([0,T];V,L^2(\Omega)).$ In fact, we have 
\[
\begin{split}
|a_V(t;u,v)|&\leq M\|u\|_V\|u\|_V\\
\Re \fra_V(t;u,u)+\omega\|u\|_{L^2(\Omega)}^2 &\geq \frac{\nu}{2} \|u\|_{V}^2,
\end{split}
 \qquad \hbox{for all }u,v\in V\hbox{ and a.e. } t\in[0,T],
 \]
where $M>0$ is a constant depending only on $\|a_{kj}\|_\infty, \|b_{k}\|_\infty, \|c_k\|_\infty$, and $\|a_0\|_\infty$, and one can choose $\omega=\sum_{k=1}^d \frac{1}{2}\big(\||\Re(b_k+c_k)|+|\Im(b_k-c_k)|\|_\infty^2\big)+\|(\Re a_0)^{-}\|_\infty$ \cite[Section 4.1]{Ouh05}. Here $(\Re a_0)^{-}=\max\{0,-\Re a_0\}.$ We can then associate a family of operators $\mathcal A_V(t)\in \L(V,V'), t\in[0,T],$ with the form $\fra_V$ which are formally given by
\begin{equation}\label{EllipOperator}\mathcal A_V(t)=-\sum_{k,j=1}^dD_j a_{kj}D_ku+\sum_{k=1}^d b_k D_k u-\sum_{k=1}^d D_k(c_ku)+c_0u.\end{equation} 
Let $A_V(t)$ be the operator associated with $a(t;\cdot,\cdot)$ on $L^2(\Omega).$ Thus $A_V(t)$ is the realization of $\mathcal A_V(t)$ in $L^2(\Omega)$ with various boundary condition which are determined by the form domain $V.$ For example $A_V(t)$ is the realization of $\mathcal A_V(t)$ with 
\begin{enumerate}[(a)]
	\item Dirichlet boundary condition if $V=H_0^1(\Omega).$ 
	\item Neumann boundary condition if $V=H^1(\Omega).$
	\item Mixed boundary condition if 
	\[V=\overline{\left\{\restr{u}{\Omega}:\ u\in C_{c}^\infty(\R^d\setminus\Gamma)\right\} }^{H^1(\Omega)}\]
	where $\Gamma $ is a closed subset of the boundary of $\Omega.$
\end{enumerate}

In particular, $\fra_V$ is associated with an evolution family $\mathcal U_V$ that governs the non-autonomous problem driven by the operator family $(\A_V(t))_{t\in [0,T]}$. Each of these evolution families is positive, it dominates the evolution family $\mathcal U_{H^1_0}$ and is dominated by $\mathcal U_{H^1}$. Following~\cite{Ouh04} we introduce the following notations:
\[
\begin{split}
f_k(t,x)&:=\sum_{j=1}^d D_j(\Im a_{kj}(t,x)),\qquad 
m(t,x):=\frac{1}{4\nu} \sum_{k=1}^d \Big[f_k(t,x)+\Im(c_k(t,x)-b_k(t,x))\Big],\\
\mathfrak{R}_V(t;u,v)&:=\int_{\Omega}\sum_{k,j=1}^d \Re (a_{kj})D_ku\overline{D_j v}\d x +\sum_{k,j=1}^d\int_{\Omega} \Big[\Re(b_k)D_ku\overline{v}\d x+\Re(c_k)u\overline{D_k v}\Big] + \int_{\Omega} \Re(a_0)u \overline{v}\d x.\end{split}
\]

\begin{lemma}\label{quasiLpkontr Neumann}
	Let $\fra_V$ be given by~\eqref{EllipOper-associated form} and denote by $\U_V$ the associated evolution family on $L^2(\Omega).$ Assume that $(|v|\wedge 1)\sgn v\in V$ for all $v\in V$. Moreover, we assume that $a_{kj}(\cdot,\cdot)$ are real-valued functions for all $k,j=1,2,\ldots d.$ Then the evolution family $\U_V$ is $L^p$-quasi-contractive for all $p\in(1,\infty[$ and we have 
	\begin{equation}\label{EllipticOpUntracontractivityEq}
	\|\U_V(t,s)f\|_{L^p(\Omega)}\leq e^{(t-s)\tilde\omega_p}\|f\|_{L^p(\Omega)} \quad \hbox{for all }(t,s)\in \Delta,
	\end{equation}
	where 
	\begin{equation} \label{omega_p Formula}\tilde\omega_p
	:=\begin{cases}
\|(\Re a_0)^{-}\|_\infty+\frac{1}{\nu}\big(\frac{1}{p}+\frac{1}{2}\big)\sum_{k=1}^d\|b_k-c_k\|_\infty^2+\frac{p}{\nu}\sum_{k=1}^d\|\Re c_k\|_\infty^2 &\hbox{ if } p\in[2,\infty[,\\
\|(\Re a_0)^{-}\|_\infty+\frac{1}{\nu}\big(\frac{1}{2}+\frac{p-1}{p}\big)\sum_{k=1}^d\|b_k-c_k\|_\infty^2+\frac{p}{\nu(p-1)}\sum_{k=1}^d\|\Re b_k\|_\infty^2 &\hbox{ if } p\in]1,2].
\end{cases}
\end{equation}
	
\end{lemma}

\begin{proof}
	The assertion follows from Theorem \ref{theorem quasilpcontractivity} and \cite[Thm.~4.3]{Ouh04}.
\end{proof}
We can also discuss the case where $a_{k,j}$ are complex-valued functions.
\begin{lemma}\label{quasiLpkontr general case}
	Let $\fra_V$ be given by~\eqref{EllipOper-associated form} such that $(|v|\wedge 1)\sgn v\in V$ for all $v\in V$ and denote by $\U_V$ the associated evolution family on $L^2(\Omega).$ Assume that $f_k\in L^\infty([0,T]\times\Omega)$, $\Im\left(a_{k,j}(t,\cdot)+a_{j,k}(t,\cdot)\right)=0$ for all $k,j=1,2,\ldots, d$ and a.e.\ $t\in[0,T]$. If either of the conditions
	
	\begin{enumerate}[(i)]
		\item $V=H_0^1(\Omega)$; 
	
	 \item $V\ne H_0^1(\Omega), (\Re u)^{+}\in V \text{ for all } u\in V$ and there exists two constants $c_1,c_2>0$ such that
	
	\[\int_{\Omega} m(t;x)|u|^2\d x\geq c_1 \int_{\Omega} |u|^2\d x+c_2\Re \mathfrak{R}_V(t;u,u)\quad u\in V \text{ and } t\in[0,T];\]
	\end{enumerate}
are satisfied, then $\U_V$ is $L^p$-quasi-contractive and \eqref{EllipticOpUntracontractivityEq} holds (up to replacing $\Re a_0(t,\cdot)$ by $\Re a_0(t,\cdot)-m$ in the expression of $\tilde	\omega_p$).
\end{lemma} 

\begin{proof}
	The assertion follows again from  Theorem \ref{theorem quasilpcontractivity} and \cite[Thm.~4.4]{Ouh04}.
\end{proof}
\begin{remark}\label{remark Linerquasi-contractivityElipOp}If $\fra_V$ fulfills the assumptions of Lemma~\ref{quasiLpkontr Neumann} or those of Lemma~\ref{quasiLpkontr general case}, then we see that the evolution family $\U_V$ is linearly quasi-contractive where \eqref{eq them:quasilpcontractivity C0S-2} is satisfied with 
\begin{eqnarray}\label{$L^p$-quasi-contractive Estimate Eq1}
\alpha_1&=&\|(\Re a_0-m)^{-}\|_\infty+\frac{1}{\nu}\sum_{k=1}^d\|b_k-c_k\|_\infty^2
\end{eqnarray}
and
\begin{eqnarray}
\label{$L^p$-quasi-contractive Estimate Eq2}\alpha_2&=&\frac{1}{\nu}\sum_{k=1}^d\|\Re c_k\|_\infty^2.
\end{eqnarray}
Likewise, the evolution family $\cev{\U^*_V}$ associated with $\cev{\fra^*_V}$ is linearly quasi-contractive and \eqref{eq them:quasilpcontractivity C0S-2} is satisfied with $\alpha_1^*=\alpha_1$ and 
\begin{equation}\label{$L^p$-quasi-contractive Estimate Eq3}
\alpha_2^*=\frac{1}{\nu}\sum_{k=1}^d\|\Re b_k\|_\infty^2.
\end{equation} 
\end{remark}

In view of Remark \eqref{remark Linerquasi-contractivityElipOp}, the following corollary follows directly from Lemma~\ref{quasiLpkontr Neumann} and Lemma~\ref{quasiLpkontr general case}.
\begin{corollary}\label{Corollary main}
Let $\fra$ be given by~\eqref{EllipOper-associated form} and denote by $\U_V$ the associated evolution family on $L^2(\Omega)$. Suppose that $V$ satisfies a Gagliardo--Nirenberg inequality and that the assumptions of Lemma~\ref{quasiLpkontr Neumann} (or those of Lemma~\ref{quasiLpkontr general case}) hold. Then $\U_V$ is ultracontractive and satisfies \ref{2Ultracontractivity for closed form} with $\mu=\nu$ and $\alpha_1=\alpha_1^*,\alpha_2, \alpha_2^*$ defined by \eqref{$L^p$-quasi-contractive Estimate Eq1}-\eqref{$L^p$-quasi-contractive Estimate Eq3}.

\end{corollary}
Now we are going to prove that the evolution family $\U_V$ governed by the time-dependent elliptic operator \eqref{EllipOperator} satisfies Gaussian bounds. We known from Theorem \ref{Thm: Gaussian bounds} that $U_V$ satisfies Gaussian bounds if and only if there exist a constants $c>0, n>0$ and $\omega\in \R$ such that 
\begin{equation*}
\|M_\rho U_V(t,s)M_\rho^{-1}\|_{\L(L^1(\R^d),L^\infty(\R^d))}\leq c(t-s)^{-\frac{n}{2}}e^{\omega(1+\rho^2)(t-s)}
\end{equation*}
for all $\rho\in\R, \psi\in W$ and $0\leq s<t\leq T.$ Let $\fra$ given by \eqref{EllipOper-associated form}. Then the non-autonomous form $\fra^{\rho}(t,u,v):=a(t;M_\rho u, M_\rho^{-1} v)$ is given by 
\begin{equation}\label{EllipOper-associated formDavies perturbation}
\begin{split}
\fra_{\rho}(t;u,v)&:=\sum_{k,j=1}^d\int_{\Omega} a_{kj}(t;x)D_ku\overline{D_j v} \ \d x\\
&\qquad+\sum_{k=1}^d\int_{\Omega} \Big[b_{k,\rho}(t;x) D_ku\overline{v} +c_{k,\rho}(t;x)u\overline{D_k v}\Big]\d x + \int_{\Omega} a_{0,\rho}(t;x)u \overline{v}\d x
\end{split}
\end{equation}
where 
\begin{align*}
b_{k,\rho}&:=b_k-\rho\sum_{j=1}^d a_{kj}D_j\psi, 
\qquad c_{k,\rho}:=c_k+\rho\sum_{i=1}^d a_{ik}D_i\psi \quad \text{ and } 
\\ a_{0,\rho}&:=a_0-\rho^2\sum_{i,k=1}^d a_{ik}D_i \psi D_k\psi+\rho \sum_{k=1}^d b_k D_k\psi-\rho \sum_{k=1}^d c_kD_k\psi.
\end{align*}
In the following we define for each $\rho\in \R, \psi\in W$ the constants $\alpha_{i,\rho}, \alpha^*_{i,\rho}, i=1,2,$ via formulas which are analogous to \eqref{$L^p$-quasi-contractive Estimate Eq1}, \eqref{$L^p$-quasi-contractive Estimate Eq2} and \eqref{$L^p$-quasi-contractive Estimate Eq3} where $\Re a_0$ is replaced by $\Re a_0-m$ if $a_{k,j}$ are complex-valued functions. Further, we set \begin{equation}\label{constant1-GaussinaEst}c_0:=\max\{\|a_{k,j}\|_\infty, \|b_k\|_\infty,\|c_k\|_\infty, \|c_0\|_\infty, k,j=1,2,\ldots,d\}\end{equation}
and 
 \begin{equation}\label{constant2-GaussinaEst}\omega:=4c_0d^2+4c_0d^3\nu^{-1}.\end{equation}
 
\begin{lemma}\label{Lemma: linQuaiLpContPerturbedEliOp}
\begin{enumerate}[(a)]
	\item For all $\rho\in \R, \psi\in W$
\begin{align}\label{ellepticity DaviesPerturbationForm}
&\Re a^\rho(t;u,u)+\omega(1+\rho^2)\|u\|^2\geq \frac{\nu}{2}\|u\|^2_V \qquad \text{ for a.e. } t\in[0,T], \text{ and all }u\in V
\end{align}
\item Assume that $f_k\in L^\infty([0,T]\times\Omega), \ \Im\big[a_{k,j}(t,\cdot)+a_{j,k}(t,\cdot)\big]=0$ for all $k,j=1,2,\ldots, d$ and $t\in[0,T].$ Then for all $\rho\in \R, \psi\in W$ we have 
\begin{align}
\alpha_{1,\rho}\label{alpha1}&=\alpha_{1,\rho}^*\leq 2\alpha_1+\rho^2(1+2d^2c_0+4d^3c_0^2\nu^{-1})+c_0d^2\\
\alpha_{2,\rho}&\label{alpha2}\leq 2\alpha_{2}+2d^3\rho^2c_0^2\nu^{-1}
\\ \label{alphareturned2}\alpha_{2,\rho}^*&\leq 2\alpha_{2}^*+2d^3\rho^2c_0^2\nu^{-1}.
\end{align}

\end{enumerate}

\end{lemma}
\begin{proof}$(a)$ We first show \eqref{ellepticity DaviesPerturbationForm}. Let $k=1,\ldots, d$ and $u\in V$ 
\begin{align*}
\Big|\Re & \Big[b_{k,\rho}(t;x) D_ku\overline{u} +c_{k,\rho}(t;x)u\overline{D_k u}\Big]\Big|
\\&\leq \Big|b_{k}(t;x) D_ku\overline{u} +c_{k}(t;x)u\overline{D_k u}\Big|+|\rho|\Big|\sum_{j=1}^da_{kj}(t;x)D_j\psi D_k u\overline{u}-\sum_{i=1}^d a_{ik}(t;x)D_i\psi u\overline{D_k u}\Big|
\\&\leq 2c_0|D_ku||u|+2d|\rho|c_0|D_ku||u|=2c_0(1+d|\rho|)|D_ku||u|
\\&\leq \frac{\nu}{2}|D_ku|^2+2c_0^2(1+d|\rho|)^2\nu^{-1}|u|^2
\\&\leq \frac{\nu}{2}|D_ku|^2+4c_0^2d^2(1+\rho^2)\nu^{-1}|u|^2,
\end{align*}
 Here we used that $|D_j\psi|<1, i=1,2,\ldots,d,$ the Young inequality and that $d\geq 1.$ Thus we have
\begin{align}\label{ellepticity DaviesPerturbationForm: Eq1}\Big|\Re \sum_{k=1}^d \Big[b_{k,\rho}(t;x) D_ku\overline{v} +c_{k,\rho}(t;x)u\overline{D_k v}\Big]\Big|\leq 
\frac{\nu}{2}\sum_{k=1}^d|D_ku|^2+4c_0^2d^3(1+\rho^2)\nu^{-1}|u|^2
\end{align}
Likewise,
\begin{align}\label{ellepticity DaviesPerturbationForm: Eq2}
\Big|\Re a_{0,\rho}(t;x)u\overline{u}\Big|&
\leq 4c_0d^2(1+\rho^2)|u|^2
\end{align}
Combining \eqref{UniformEllipticityEq}, \eqref{ellepticity DaviesPerturbationForm: Eq1} and \eqref{ellepticity DaviesPerturbationForm: Eq2} yields \eqref{ellepticity DaviesPerturbationForm}.

$(b)$ Using again that $|D_j\psi|<1, i=1,2,\ldots,d,$ one easily prove \eqref{alpha2} and \eqref{alphareturned2}.  Further,  
\begin{align}
\nonumber\frac{1}{\nu}\sum_{k}\|b_{k,\rho}-c_{k,\rho}\|_\infty^2&=\frac{1}{\nu}\sum_{k}\Big[\|b_k-a_k-\rho\sum_{j=1}^d a_{kj}D_j\psi-\rho\sum_{i=1}^d a_{ik}D_i\psi\|_\infty^2\Big]
\\&\nonumber\leq \frac{1}{\nu}\sum_{k}\Big[2\|b_{k,}-a_{k}\|_\infty^2+4d^2\rho^2c_0^2\Big]
\\&\label{alpha1proofEq1}\leq \frac{2}{\nu}\sum_{k}\|b_{k,}-a_{k}\|_\infty^2+\frac{1}{\nu}4d^3\rho^2c_0^2
\end{align}
Since $\Im(a_{k,j}+a_{j,k})=0$ for all $k,j=1,2,\ldots, d,$
 we deduce that $m_\rho=m$ for every $\rho\in\R.$ It follows that 
\begin{align*}
\|(\Re a_{0,\rho}-m_\rho)^{-}\|_\infty&\leq \|(\Re a_{0}-m)^{-}\|_\infty+\|\rho^2\sum_{i,k=1}^d a_{ik}D_i \psi D_k\psi+\rho \sum_{k=1}^d b_k D_k\psi-\rho \sum_{k=1}^d c_kD_k\psi\|_\infty
\\&\leq \|(\Re a_{0}-m)^{-}\|_\infty+\rho^2d^2c_0+\rho^2+c_0d^2.
\end{align*}
This equality together with \eqref{alpha1proofEq1} prove \eqref{alpha1}.
\end{proof}

Combining Theorem~\ref{Main 2 TheoremGaussianEstimate} with Lemma \ref{Lemma: linQuaiLpContPerturbedEliOp} and Corollary \ref{Corollary main} we can finally prove  Gaussian bounds for evolution families associated with families of uniform elliptic operators of the form \eqref{EllipOperator}.

\begin{theorem}\label{Main TheoremGaussianEstimateEllipOp} Let $V$ be $W-$invariant and satisfy \eqref{Gagliardo--Nirenberg inequality}. If the assumptions of Lemma~\ref{quasiLpkontr Neumann} or those of Lemma \ref{quasiLpkontr general case} are satisfied, then $\U_V$ satisfies Gaussian bounds. More precisely we have
	 \begin{equation}
	\label{eq2: pseudoGaussian boundsEllOp}
	(U_V(t,s)f)(x)=\int_{\R^d} \Gamma_V(t,s,x,y)f(y)d y\quad
	\end{equation}
	where 
		\begin{equation}
	\label{eq: pseudoGaussian boundsEllOp}
	|\Gamma_V(t,s,x,y)|\displaystyle\leq c_0e^{\omega_0 (t-s)}(t-s)^{-\frac{n}{2}} \exp\Big(-\frac{\beta_0}{4}\frac{|x-y|^2}{t-s}\Big)
	\end{equation}
for a.e.\ $x\in\R^d$, all $(t,s)\in \Delta$, and all $f\in L^2(\R^d)$, where 
$c, \omega_0, n$ and $\beta_0$ are positive constants that depend only on $C_G, N, d,\nu, c_0$ and on the constant $\beta$ defined in \eqref{eq Definition constant beta}.
\end{theorem}

\bibliographystyle{unsrt}
\bibliography{../../referenzen/literatur}

\end{document}